\documentclass[12pt, a4paper]{amsart}
\usepackage{amscd,amssymb,mathtools,bbold,graphics}
\usepackage{amsmath}

\usepackage{float}

\usepackage{enumitem}
\usepackage{tensor}
\usepackage{wrapfig}
\usepackage{tikz}

\allowdisplaybreaks

\usetikzlibrary{cd}
\usetikzlibrary{ext.paths.ortho,shapes,fit,positioning}
\usetikzlibrary{calc}

\tikzset{tube/.style={draw, ellipse, minimum width=1.5cm}}

\usepackage{color, url, hyperref}

\hypersetup{colorlinks=true, linkcolor=blue, linkbordercolor=blue, citecolor=blue, citebordercolor=blue, linktocpage=true}

\def\Aut{\operatorname{Aut}} 
\def\End{\operatorname{End}}

\def\dim{\operatorname{dim}}

\def\id{\operatorname{id}}

\def\NC{\operatorname{NC}}

\def\Hom{\operatorname{Hom}}
\def\spn{\operatorname{span}}
\def\Rep{\operatorname{Rep}}
\def\dom{\operatorname{dom}}
\def\cod{\operatorname{cod}}

\def\A{\mathbb{A}}

\def\C{\mathbb{C}}

\def\N{\mathbb{N}}

\def\BB{\mathcal{B}}
\def\CC{\mathcal{C}}

\def\HH{\mathcal{H}}

\def\KK{\mathcal{K}}

\def\RR{\mathcal{R}}

\def\TT{\mathcal{T}}

\newcommand{\idpart}{\hspace{0.5ex}\raisebox{-0.25ex}{\begin{tikzpicture}[bullet/.style={circle, fill, minimum size=2pt, inner sep=0pt, outer sep=0pt},nodes=bullet]
\node (t) at (0,0) {};
\node (b) at (0,-1.5ex) {};
\draw (t) -- (b);
\end{tikzpicture}
}}

\newcommand{\pants}{\hspace{0.5ex}\raisebox{-0.25ex}{\begin{tikzpicture}[bullet/.style={circle, fill, minimum size=2pt, inner sep=0pt, outer sep=0pt},nodes=bullet]
\node (t) at (0,0) {};
\node (b1) at (-0.5ex,-1.5ex) {};
\node (b2) at (0.5ex,-1.5ex) {};
\draw (t) -- (0,-0.75ex) -- (-0.5ex,-0.75ex) -- (b1);
\draw (0,-0.75ex) -- (0.5ex,-0.75ex) -- (b2);
\end{tikzpicture}
}}

\newcommand{\shirt}{\hspace{0.5ex}\raisebox{-0.25ex}{\begin{tikzpicture}[bullet/.style={circle, fill, minimum size=2pt, inner sep=0pt, outer sep=0pt},nodes=bullet]
\node (t1) at (-0.5ex,0) {};
\node (t2) at (0.5ex,0) {};
\node (b) at (0,-1.5ex) {};
\draw (b) -- (0,-0.75ex) -- (-0.5ex,-0.75ex) -- (t1);
\draw (0,-0.75ex) -- (0.5ex,-0.75ex) -- (t2);
\end{tikzpicture}
}}

\newcommand{\cappart}{\hspace{0.5ex}\raisebox{-0.25ex}{\begin{tikzpicture}[bullet/.style={circle, fill, minimum size=2pt, inner sep=0pt, outer sep=0pt},nodes=bullet]
\node (b) at (0,-1.5ex) {};
\draw (b) -- (0,-0.75ex);
\draw (-0.25ex,-0.75ex) -- (0.25ex,-0.75ex);
\end{tikzpicture}
}}

\newcommand{\cuppart}{\hspace{0.5ex}\raisebox{0.65ex}{\begin{tikzpicture}[bullet/.style={circle, fill, minimum size=2pt, inner sep=0pt, outer sep=0pt},nodes=bullet]
\node (t) at (0,0) {};
\draw (t) -- (0,-0.75ex);
\draw (-0.25ex,-0.75ex) -- (0.25ex,-0.75ex);
\end{tikzpicture}
}}

\newcommand{\paircup}{\hspace{0.5ex}\raisebox{0.65ex}{\begin{tikzpicture}[bullet/.style={circle, fill, minimum size=2pt, inner sep=0pt, outer sep=0pt},nodes=bullet]
\node (t1) at (-0.5ex,0) {};
\node (t2) at (0.5ex,0) {};
\draw (t1) -- (-0.5ex,-0.75ex) -- (0.5ex,-0.75ex) -- (t2);
\end{tikzpicture}
}}

\newcommand{\paircap}{\hspace{0.5ex}\raisebox{-0.25ex}{\begin{tikzpicture}[bullet/.style={circle, fill, minimum size=2pt, inner sep=0pt, outer sep=0pt},nodes=bullet]
\node (b1) at (-0.5ex,0) {};
\node (b2) at (0.5ex,0) {};
\draw (b1) -- (-0.5ex,0.75ex) -- (0.5ex,0.75ex) -- (b2);
\end{tikzpicture}
}}

\newcommand{\tricky}{\hspace{0.5ex}\raisebox{0.65ex}{\begin{tikzpicture}[bullet/.style={circle, fill, minimum size=2pt, inner sep=0pt, outer sep=0pt},nodes=bullet]
			\node (t1) at (-1ex,0) {};
			\node (t) at (0,0) {};
			\draw (t) -- (0,-0.7ex);
			\draw (-0.25ex,-0.7ex) -- (0.25ex,-0.7ex);
			\node (t2) at (1ex,0) {};
			\draw (t1) -- (-1ex,-1.2ex) -- (1ex,-1.2ex) -- (t2);
		\end{tikzpicture}
}}

\newtheorem{thm}{Theorem}[section]

\newtheorem{lemma}[thm]{Lemma}
\newtheorem{prop}[thm]{Proposition}

\theoremstyle{definition}
\newtheorem{definition}[thm]{Definition}
\newtheorem{notation}[thm]{Notation}
\theoremstyle{remark}
\newtheorem{remark}[thm]{Remark}

\newtheorem{example}[thm]{Example}

\numberwithin{equation}{section}

\tikzstyle{vertex}=[circle]
\tikzstyle{goto}=[->,shorten >=1pt,>=stealth,semithick]

\begin{document}

\date{\today}
\title[Notes on tubes]{Tubular partitions and representations of the quantum automorphism group of a homogeneous rooted tree}
\author{Nathan Brownlowe}
\address{Nathan Brownlowe, School of Mathematics and Statistics, University of Sydney,
NSW 2006, Australia.}
\email{nathan.brownlowe@sydney.edu.au}

\author{David Robertson}
\address{David Robertson, School of Science and Technology, University of New England, NSW 2351, Australia.}
\email{david.robertson@une.edu.au}

 \thanks{Brownlowe was supported by the Australian Research Council grant DP200100155.}

\subjclass[2010]{}
\thanks{}

\begin{abstract}
We introduce the rigid tensor category of tubular partitions, and use it to provide a combinatorial model for the representation category of the quantum automorphism group of a homogeneous rooted tree.
\end{abstract}

\maketitle

\setcounter{tocdepth}{1}


\section{Introduction}

Pontryagin duality says that a compact abelian group $G$ can be reconstructed from its dual group $\hat{G}$. Classical Tannaka-Krein duality extends this idea to non-abelian compact groups $G$, where the dual group is replaced by the category $\Rep(G)$ of finite-dimensional unitary representations of $G$. While $\Rep(G)$ is not in general a group, the structure of $\Rep(G)$ can be used to reconstruct $G$.

Classical Tannaka-Krein duality has been extended in various directions, including to the setting of compact quantum groups. In \cite{Wor87} Woronowicz introduced compact quantum groups as unital $C^*$-algebras equipped with a coassociative comultiplication and satisfying certain density conditions. The definition is such that a \textit{commutative} compact quantum group is just the $C^*$-algebra of continuous functions on a compact group. Tannaka-Krein duality was developed by Woronowicz in \cite{Wor88} for compact \textit{matrix} quantum groups, showing that any compact matrix quantum group can be reconstructed from the category of its finite-dimensional unitary representations, which is now a rigid $C^*$-tensor category that is non-symmetric. The non-symmetry of this category reflects the non-commutativity of the underlying $C^*$-algebra. In \cite{Wang97} Wang showed that Woronowicz's results extend naturally to cover all compact quantum groups. 

An important class of compact quantum groups are the quantum symmetric groups $S_n^+$ introduced by Wang in \cite{Wang1998} in his quest to find the correct notion of a quantum permutation of $n$ points. The underlying $C^*$-algebra of $S_n^+$ is the universal $C^*$-algebra generated by the entries of a magic unitary matrix---an $n \times n$ matrix of orthogonal projections where each row and column sum to $1$ (meaning $S_n^+$ is a compact matrix quantum group). When $n \leq 3$, $S_n^+$ is commutative and agrees with the classical symmetric group $S_n$; for $n \ge 4$, $S_n^+$ is non-classical and is infinite-dimensional as a $C^*$-algebra. 

The representation category of $S_n^+$ is very well understood, and forms part of the study of easy quantum groups. In \cite{BS2009} Banica and Speicher introduced a class of compact quantum groups whose representation category can be described by a category of partitions; these quantum groups are now commonly known as easy quantum groups. Symmetry of this representation category corresponds to the category of partitions containing a crossing partition, and the corresponding $C^*$-algebras are commutative. Any non-crossing category of partitons therefore corresponds to a quantum group, and the category of \textit{all} non-crossing partitions corresponds to the quantum symmetric groups of Wang. The wider class of easy quantum groups have given rise to a number of new examples of compact quantum groups, and their classification was completed in \cite{RW2016} after work in \cite{BS2009,BCP2010,Web2013,RW2014,RW2015}.

The notion of a quantum permutation of $n$ points leads naturally to questions of quantum analogues of automorphisms of more general objects. For $G$ a finite graph, the quantum automorphism group $\Aut^+(G)$ of $G$ was introduced by Bichon in \cite{Bichon03}, and is the universal compact quantum group coacting on the vertex set of $G$ in an adjacency preserving way. The underlying $C^*$-algebras can again be described as a universal $C^*$-algebra by imposing additional relations on the magic unitary defining the quantum symmetric group of the vertex set. The problem of whether a given finite graph has quantum symmetry or not, or in other words, whether the underlying $C^*$-algebra is commutative, is still open in full generality. It is known that almost all graphs have no quantum symmetry \cite{LMR2020}, which is a quantum analogue of the classical result that in a probabalistic sense almost all graphs have no classical symmetry. 

While it is uncommon for a graph to admit quantum symmetry, it was shown in \cite{JSW2020} that almost all trees admit quantum symmetry. A special class of trees was considered by the authors in \cite{BR2025}, where we examined the quantum automorphisms of infinite rooted homogeneous trees. In particular, for a finite set $X$ we introduced the quantum automorphism group $\A_X$ of the infinite rooted homogeneous tree $X^*$,  and we used $\A_X$ to introduce the notion of self-similarity for compact quantum groups.  

%

This paper is devoted to studying the representation category of $\A_X$. In \cite{BR2025}, $\A_X$ is shown to be an infinitely iterated free wreath product of the quantum symmetric group $S_X^+$, giving a quantum analogue of the description of the classical automorphism group $\Aut(X^*)$ as an infinitely iterated wreath product of the symmetric group $S_X$. This suggests that the structure of the representation category of $\A_X$ should be governed by combinatorial data involving categories of partitions, as is the case for quantum symmetric groups. Indeed, in \cite{LT2016} Lemeux and Tarrago describe the representation category of a free wreath product of a compact quantum group and a quantum symmetric group as a category of decorated partitions. In this paper we define a higher-dimensional partition we call a \emph{tubular partition}, and we show that the category of all tubular partitions determines the representation category of $\A_X$.

In section \ref{sec:prelims} we recall the necessary definitions and results about compact quantum groups and their representation categories, including Woronowicz's (and Wang's) notion of Tannaka-Krein duality. In section \ref{sec:tubularpartitions} we define the category of tubular partitions and we look at its structure. In section \ref{sec:tkduality} we state our main result that the rigid $C^*$-tensor category determined by the cateogry of tubular partitions is the representation category of $\A_X$.

\section{Preliminaries} \label{sec:prelims}

In this section we state the required definitions and results concerning representations of compact quantum groups, as well as the definition of the quantum automorphism group $\mathbb{A}_X$ of of a homogeneous rooted tree.  

\subsection{Compact quantum groups and their representations}

The following notion of a compact quantum group was introduced in \cite{Wor95}. 

\begin{definition} \cite[Definition 1]{Wor95} \label{CQG:densitydefn}
	A \emph{compact quantum group} is a pair $(A,\Delta)$ where $A$ is a unital $C^*$-algebra and $\Delta : A \to A \otimes A$ is a unital $*$-homomorphism such that
	\begin{enumerate}
		\item \label{cqg:coassoc} $(\Delta\otimes \id)\Delta = (\id\otimes\Delta) \Delta$
		\item $\overline{(A\otimes 1)\Delta(A)} = A \otimes A = \overline{(1\otimes A)\Delta(A)}$.
	\end{enumerate}
	We call $\Delta$ the \emph{comultiplication} and (\ref{cqg:coassoc}) is called \emph{coassociativity}.
\end{definition}

We include the following definitions for completeness; more details about the representation theory of compact quantum groups can be found in \cite{NeshTuset}.

\begin{definition}
Let $(A,\Delta)$ be a compact quantum group. A \emph{finite-dimensional unitary representation} $u$ of $A$ is  a finite-dimensional Hilbert space $\HH$ and a unitary $u\in \BB(\HH)\otimes A$ satisfying (using the standard leg-numbering notation)
\begin{equation} \label{CQGrep}
(\id\otimes\Delta)(u) = u_{12}u_{13} \in \BB(\HH)\otimes A \otimes A.
\end{equation}
\end{definition}

If we choose a basis $\{e_i : 1\leq i \leq d_u:=\dim(\HH)\}$ and identify $\BB(\HH)\otimes A$ with $M_{d_u}(A)$ then $u = (u_{i,j})_{i,j=1}^{d_u}$ and (\ref{CQGrep}) becomes the identity
\[
 \Delta(u_{i,j}) = \sum_{k=1}^{d_u} u_{i,k}\otimes u_{k,j}.
\]
\begin{definition}
Let $u$ and $v$ be finite-dimensional unitary representations of a compact quantum group $(A,\Delta)$ on Hilbert spaces $\HH$ and $\KK$. An \emph{intertwiner} from $u$ to $v$ is a linear map $T\in\BB(\HH,\KK)$ satisfying
\[
 (T\otimes 1_A)u = v(T\otimes 1_A).
\]
\end{definition}

\subsection{The category of finite-dimensional unitary representations $\Rep(A)$}

The objects of $\Rep(A)$ are finite-dimensional unitary representations $u$ of $A$, and a morphism from $u$ to $v$ is an intertwiner $T\in\BB(\HH,\KK)$. The identity morphism $\id_u$ is the identity intertwiner $I\in\BB(\HH)$ and composition of morphisms is the usual composition of linear maps. Given $u,v\in\Rep(A)$, the tensor product is
\[
 u\otimes v = u_{13}v_{23} \in \BB(\HH\otimes\KK)\otimes A.
\]
This turns $\Rep(A)$ into a $C^*$-tensor category. In fact, $\Rep(A)$ is a rigid $C^*$-tensor category, where the dual of $u$ is the contragredient representation $u^c$.


\subsection{Tannaka-Krein duality for compact quantum groups}

\begin{definition}\label{def: model}
Let $\RR$ be a complete concrete rigid $C^*$-tensor category. A pair $M = (B,\{v^\alpha\}_{\alpha\in\RR})$ is a \emph{model} of $\RR$ if
\begin{itemize}
 \item[(i)] $B$ is a unital $C^*$-algebra
 \item[(ii)] for each $\alpha\in\RR$, $v^\alpha$ is a unitary element of $\BB(\HH_\alpha)\otimes B$ where $\HH_\alpha$ is a finite-dimensional Hilbert space
 \item[(iii)] $v^{\alpha\otimes\beta} = v^\alpha\otimes v^\beta$ for all $\alpha,\beta\in\RR$
 \item[(iv)] $(T\otimes 1_B)v^\alpha = v^\beta(T\otimes 1_B)$ for all $\alpha,\beta\in\RR$ and $T\in\Hom(\alpha,\beta)$.
\end{itemize}
\end{definition}

\begin{definition}
Suppose $\RR$ is generated by a set of objects $N\subset\RR$. A pair $(B,\{u^\alpha\}_{\alpha\in N})$ is called \emph{$\RR$-admissible} if there is a model $M = (B,\{v^\alpha\}_{\alpha\in\RR})$ with $u^\alpha = v^\alpha$ for all $\alpha \in N$. An $\RR$-admissible pair $(B,\{u^\alpha\}_{\alpha\in N})$ is called \emph{universal} if, for any other $\RR$-admissible pair $(C,\{w^\alpha\}_{\alpha\in N})$ there exists a $C^*$-homomorphism $\varphi:B\to C$ such that
\[
 (\id\otimes\varphi)u^\alpha = w^\alpha
\]
for all $\alpha\in N$.
\end{definition}

The following result was first proved in \cite{Wor88} for compact matrix quantum groups, and then extended to compact quantum groups in \cite[A.3. Theorem]{Wang97}. 

\begin{thm}\label{thm:Wang}
Let $\RR$ be a complete concrete rigid $C^*$-tensor category generated by $N\subset \RR$. Then there is a unique universal $\RR$-admissible pair $(A,\{u^\alpha\}_{\alpha\in N})$ with uniquely determined model $(A,\{u^\alpha\}_{\alpha\in \RR})$ such that $A$ is a compact quantum group with comultiplication $\Delta$, and for all $\alpha\in\RR$, $u^\alpha$ is a finite-dimensional unitary representation of $(A,\Delta)$ on $\HH_\alpha$.
\end{thm}

\subsection{Quantum automorphisms of a rooted homogenous tree}
Let $X$ be a finite set, and consider the infinite homogeneous rooted tree $X^*$. Recall from \cite[Definition~3.2]{BR2025} and \cite[Theorem~3.4]{BR2025} that the quantum automorphism group $X^*$ is the compact quantum group $(\A_X,\Delta)$, where  $\mathbb{A}_X$ is the universal $C^*$-algebra generated by elements $\{a_{u,v}: u,v\in X^n, n\geq 0\}$ subject to the relations
	\begin{enumerate}
		\item \label{qaut:identity} $a_{\varnothing,\varnothing} = 1$,
		\item \label{qaut:proj} for any $n\geq 0, u,v\in X^n$, $a_{u,v}^* = a_{u,v}^2 = a_{u,v}$, and
		\item \label{qaut:sum} for any $n\geq 0, u,v\in X^n$ and $x\in X$,
		\[
		a_{u,v} = \sum_{y\in X} a_{ux,vy} = \sum_{z\in X} a_{uz,vx}.
		\]
	\end{enumerate}
The comultiplication $\Delta : \A_X \to \A_X\otimes \A_X$ satisfies
	\[
	\Delta(a_{u,v}) = \sum_{w\in X^n} a_{u,w} \otimes a_{w,v},
	\]
	for all $u,v\in X^n$ and $n\geq 1$

\begin{remark}\label{rmk:consequencefromAX}
	Note the following consequences of the defining relations of $\A_X$:
	\begin{itemize}
		\item[(i)] As is pointed out in \cite[Remarks~3.3]{BR2025}, we have $a_{x,z}a_{y,z}=\delta_{x,y}a_{x,z}$ for all $x,y,z\in X$. 
		\item[(ii)] If we take $u,v=\emptyset$ in relation (3) above, then we get $\sum_{y\in X}a_{x,y}=\sum_{x\in X}a_{x,y}=1_{\A_X}$ for all $x,y\in X$. 
	\end{itemize}
\end{remark}

\begin{notation}\label{note: a_n}
For $n\in\N$ we write $a^{(n)} := (a_{u,v})_{u,v\in X^n}$. Note that $a^{(0)}=(1_{\A_X})$. It follows from the defining relations of $\A_X$ that each $a^{(n)}\in \Rep(\A_X)$ with $(a^{(n)})^*=(a_{v,u})_{u,v\in X^n}$.  
	\end{notation}
	

\section{The category of tubular partitions} \label{sec:tubularpartitions}

In this section we define a category $\TT$, the morphisms of which we will call \emph{tubular partitions.} This category will be an increasing union of small categories $\CC_k$ for $k\geq 1$, identified via embedding functors. Defining the categories $\CC_k$ is done through induction.

\subsection{The base category}

\subsubsection{Category structure} \label{subsec:catstructure}

The base category $\CC_1$ is the category of \emph{non-crossing partitions}, typically denoted as $\NC$. We will now recap the key definitions we need about this category, but further details can be found in \cite{BS2009}.

The objects in $\NC$ are positive integers $n\geq 0$ and a morphism $p\in\NC(m,n)$ is a partition of $m+n$ points $\{1,\dots,m+n\}$ into blocks, such that there are no points $x_1<x_2<x_3<x_4$ with $\{x_1,x_3\}$ and $\{x_2,x_4\}$ contained in distinct blocks. We write $p = \{B_1,\dots,B_l\}$ where each $B_i \subseteq \{1,\dots,m+n\}$. The upper and lower points in a block $B_i$ are denoted by $U_{B_i} = B_i\cap\{1,\dots,m\}$ and $L_{B_i} = B_i\cap\{m+1,\dots,m+n\}$, respectively.

For example, consider $p = \{\{1,5\},\{2\},\{3,4\}\} \in \NC(2,3)$. We can depict $p$ graphically by drawing $2$ upper points and $3$ lower points, then joining points in the same block with a line:
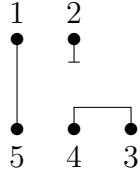
\begin{center}
\begin{tikzpicture}[scale=1.5,every node/.style={inner sep=-1}]
  \def\dx{0.5}
  \def\dy{0.2} 
  \node[label={[label distance=5]above:1}] (t1) at (0,0) {$\bullet$};
  \node[label={[label distance=5]above:2}] (t2) at (\dx,0) {$\bullet$};
  \node[label={[label distance=5]below:5}] (b5) at (0,-4*\dy) {$\bullet$};
  \node[label={[label distance=5]below:4}] (b4) at (\dx,-4*\dy) {$\bullet$};
  \node[label={[label distance=5]below:3}] (b3) at (2*\dx,-4*\dy) {$\bullet$};
  \draw (t1) -- (b5);
  \draw (t2) -- (\dx,-\dy);
  \draw (\dx- 0.05,-\dy) -- (\dx+ 0.05,-\dy);
  \draw (b4) -- (\dx,-3*\dy) -- (2*\dx,-3*\dy) -- (b3);
\end{tikzpicture}
\end{center}
Note that the upper points are labelled left-to-right while the lower points are labelled right-to-left. This ensures that the lines joining points in a non-crossing partition are visually non-crossing. With this convention set, we will not continue to label points in partitions.

Composition in $\NC$ is best described graphically: given $p\in\NC(m,n)$ and $q\in\NC(n,r)$ the composition $q\circ p\in\NC(m,r)$ is the partition obtained by first stacking the diagram representing $p$ on top of the diagram representing $q$, and then identifying the middle $n$ points and deleting any block which has no connection to the $m$ upper or $r$ lower points. For example, given $p$ above and $q = \{\{1,4,5,7\},\{2,3\},\{6\}\} \in \NC(3,4)$ we can depict the composition as follows:
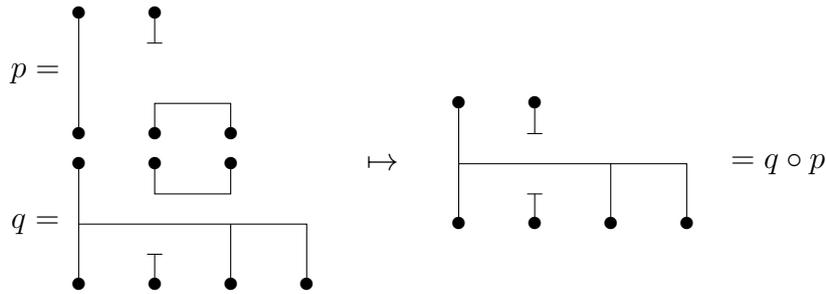
\begin{center}
\begin{tikzpicture}[scale=2,every node/.style={inner sep=-1}]
  \def\dx{0.5}
  \def\dy{0.2}
\begin{scope}   
  \node (p) at (-0.5*\dx,-2*\dy) {$p = \ $};
  \node (t1) at (0,0) {$\bullet$};
  \node (t2) at (\dx,0) {$\bullet$};
  \node (b5) at (0,-4*\dy) {$\bullet$};
  \node (b4) at (\dx,-4*\dy) {$\bullet$};
  \node (b3) at (2*\dx,-4*\dy) {$\bullet$};
  \draw (t1) -- (b5);
  \draw (t2) -- (\dx,-\dy);
  \draw (\dx- 0.05,-\dy) -- (\dx+ 0.05,-\dy);
  \draw (b4) -- (\dx,-3*\dy) -- (2*\dx,-3*\dy) -- (b3);
\end{scope}
\begin{scope}[shift={(0,-5*\dy)}]
  \node (q) at (-0.5*\dx,-2*\dy) {$q = \ $};
  \node (t1) at (0,0) {$\bullet$};
  \node (t2) at (\dx,0) {$\bullet$};
  \node (t3) at (2*\dx,0) {$\bullet$};
  \node (b7) at (0,-4*\dy) {$\bullet$};
  \node (b6) at (\dx,-4*\dy) {$\bullet$};
  \node (b5) at (2*\dx,-4*\dy) {$\bullet$};
  \node (b4) at (3*\dx,-4*\dy) {$\bullet$};
  \draw (t1) -- (b7);
  \draw (t2) -- (\dx,-\dy) -- (2*\dx,-\dy) -- (t3);
  \draw (0,-2*\dy) -- (3*\dx,-2*\dy) -- (b4);
  \draw (2*\dx,-4*\dy) -- (2*\dx,-2*\dy);
  \draw (b6) -- (\dx,-3*\dy);
  \draw (\dx- 0.05,-3*\dy) -- (\dx+ 0.05,-3*\dy);
\end{scope}
\begin{scope}[shift={(5*\dx,-3*\dy)}]
  \node (arrow) at (-\dx,-2*\dy) {$\mapsto$};
  \node (t1) at (0,0) {$\bullet$};
  \node (t2) at (\dx,0) {$\bullet$};
  \node (b7) at (0,-4*\dy) {$\bullet$};
  \node (b6) at (\dx,-4*\dy) {$\bullet$};
  \node (b5) at (2*\dx,-4*\dy) {$\bullet$};
  \node (b4) at (3*\dx,-4*\dy) {$\bullet$};
  \node (qp) at (4.2*\dx,-2*\dy) {$= q\circ p$};
  \draw (t2) -- (\dx,-\dy);
  \draw (\dx- 0.05,-\dy) -- (\dx+ 0.05,-\dy);
  \draw (t1) -- (b7);
  \draw (0,-2*\dy) -- (3*\dx,-2*\dy) -- (b4);
  \draw (2*\dx,-4*\dy) -- (2*\dx,-2*\dy);
  \draw (b6) -- (\dx,-3*\dy);
  \draw (\dx- 0.05,-3*\dy) -- (\dx+ 0.05,-3*\dy);
\end{scope}
\end{tikzpicture}
\end{center}

We see that $q\circ p = \{\{1,3,4,6\},\{2\},\{5\}\} \in \NC(2,4)$.

In what follows, we will frequently make use of the following single-block partitions in $\NC$:
\begin{itemize}
\item $\idpart \in\NC(1,1)$, which is also the identity morphism  $\id_1$,
\item $\cuppart\in\NC(1,0)$ and $\cappart\in\NC(0,1)$,
\item $\shirt\in\NC(2,1)$ and $\pants\in \NC(1,2)$, and 
\item $\paircup\in\NC(2,0)$ and $\paircap\in \NC(0,2)$.
\end{itemize}

Any partition $p \in \NC(m,n)$ has an \emph{adjoint} $p^* \in \NC(n,m)$ obtained by reflecting $p$ with respect to a horizontal axis between $m$ and $n$. For example, with the same $p$ as defined above we have
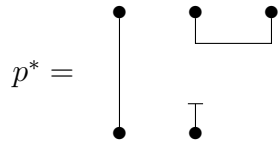
\begin{center}
\begin{tikzpicture}[scale=2,every node/.style={inner sep=-1}]
  \def\dx{0.5}
  \def\dy{0.2}   
  \node (p) at (-\dx,-2*\dy) {$p^* = $};
  \node (t1) at (0,-4*\dy) {$\bullet$};
  \node (t2) at (\dx,-4*\dy) {$\bullet$};
  \node (b5) at (0,0) {$\bullet$};
  \node (b4) at (\dx,0) {$\bullet$};
  \node (b3) at (2*\dx,0) {$\bullet$};
  \draw (t1) -- (b5);
  \draw (t2) -- (\dx,-3*\dy);
  \draw (\dx- 0.05,-3*\dy) -- (\dx+ 0.05,-3*\dy);
  \draw (b4) -- (\dx,-\dy) -- (2*\dx,-\dy) -- (b3);
  
\end{tikzpicture}
\end{center}

\subsubsection{Monoidal structure and rigidity}

The monoidal structure on $\CC_1=\NC$ is given on objects by $m\otimes n := m+n$, and for morphisms $p$ and $q$, $p\otimes q$ is graphically represented by drawing $p$ and $q$ horizontally adjacent. For example, given the same $p$ and $q$ as above we have
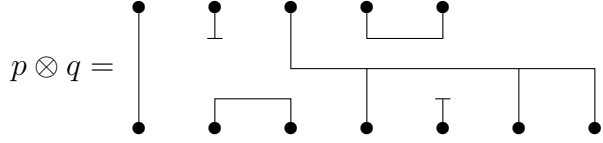
\begin{center}
\begin{tikzpicture}[scale=2,every node/.style={inner sep=-1}]
  \def\dx{0.5}
  \def\dy{0.2}
\begin{scope}
  \node (potimesq) at (-\dx,-2*\dy) {$p\otimes q = $};
  \node (t1) at (0,0) {$\bullet$};
  \node (t2) at (\dx,0) {$\bullet$};
  \node (b5) at (0,-4*\dy) {$\bullet$};
  \node (b4) at (\dx,-4*\dy) {$\bullet$};
  \node (b3) at (2*\dx,-4*\dy) {$\bullet$};
  \draw (t1) -- (b5);
  \draw (t2) -- (\dx,-\dy);
  \draw (\dx- 0.05,-\dy) -- (\dx+ 0.05,-\dy);
  \draw (b4) -- (\dx,-3*\dy) -- (2*\dx,-3*\dy) -- (b3);
\end{scope}
\begin{scope}[shift={(2*\dx,0)}]
  \node (t1) at (0,0) {$\bullet$};
  \node (t2) at (\dx,0) {$\bullet$};
  \node (t3) at (2*\dx,0) {$\bullet$};
  \node (b7) at (\dx,-4*\dy) {$\bullet$};
  \node (b6) at (2*\dx,-4*\dy) {$\bullet$};
  \node (b5) at (3*\dx,-4*\dy) {$\bullet$};
  \node (b4) at (4*\dx,-4*\dy) {$\bullet$};
  \draw (t1) -- (0,-2*\dy) -- (\dx,-2*\dy) -- (b7);
  \draw (t2) -- (\dx,-\dy) -- (2*\dx,-\dy) -- (t3);
  \draw (\dx,-2*\dy) -- (4*\dx,-2*\dy) -- (b4);
  \draw (3*\dx,-4*\dy) -- (3*\dx,-2*\dy);
  \draw (b6) -- (2*\dx,-3*\dy);
  \draw (2*\dx- 0.05,-3*\dy) -- (2*\dx+ 0.05,-3*\dy);
\end{scope}
\end{tikzpicture}
\end{center}
Hence $p\otimes q = \{\{1,12\},\{2\},\{3,6,7,9\},\{4,5\},\{8\},\{10,11\}\} \in \NC(5,7)$.

We will not define the notion of a \textit{rigid} monoidal category, but we remark that objects in $\NC$ are self-dual, and for every $m\in \NC$ the morphisms giving rigidity are $\eta_m\in \CC_1(0,2m), \epsilon_m \in \CC_1(2m,0)$ given by

\begin{center}
\vspace{1ex}
\begin{tikzpicture}[scale=1.2,every node/.style={inner sep=-1}]
  \def\dx{0.5}
  \def\dy{0.2}
\begin{scope}
  \node (eta) at (-1.5*\dx,1.5*\dy) {$\eta_m=$};
  \node (1) at (0,0) {$\bullet$};
  \node (2) at (\dx,0) {$\bullet$};
  \node (dots1) at (2*\dx,0) {$\dots$};
  \node (3) at (3*\dx,0) {$\bullet$};
  \node (4) at (4*\dx,0) {$\bullet$};
  \node (dots2) at (5*\dx,0) {$\dots$};
  \node (5) at (6*\dx,0) {$\bullet$};
  \node (6) at (7*\dx,0) {$\bullet$};
  
  \draw (1) to (0,3*\dy) to (7*\dx,3*\dy) to (6);
  \draw (2) to (\dx,2*\dy) to (6*\dx,2*\dy) to (5);
  \draw (3) to (3*\dx,\dy) to (4*\dx,\dy) to (4);
\end{scope}
\begin{scope}[shift={(12*\dx,3*\dy)}]
  \node (eps) at (-1.5*\dx,-1.5*\dy) {$\epsilon_m=$};
  \node (1) at (0,0) {$\bullet$};
  \node (2) at (\dx,0) {$\bullet$};
  \node (dots1) at (2*\dx,0) {$\dots$};
  \node (3) at (3*\dx,0) {$\bullet$};
  \node (4) at (4*\dx,0) {$\bullet$};
  \node (dots2) at (5*\dx,0) {$\dots$};
  \node (5) at (6*\dx,0) {$\bullet$};
  \node (6) at (7*\dx,0) {$\bullet$};
  
  \draw (1) to (0,-3*\dy) to (7*\dx,-3*\dy) to (6);
  \draw (2) to (\dx,-2*\dy) to (6*\dx,-2*\dy) to (5);
  \draw (3) to (3*\dx,-\dy) to (4*\dx,-\dy) to (4);
\end{scope}
\end{tikzpicture}
\vspace{1ex}
\end{center}

\subsection{Induction}

\subsubsection{Category structure}

Suppose $\CC_k$ is given and define the objects of a category $\CC_{k+1}$ by
\[
 \CC_{k+1} = \{ (m;\alpha_1,\dots,\alpha_m) : m\in\CC_1, \alpha_i\in\CC_k \text{ for }1\le i\le m\}.
\]
Given objects $\gamma = (m;\alpha_1,\dots,\alpha_m)$ and $\delta=(n;\beta_1,\dots,\beta_n)$ a morphism between them is $(p;\phi_1,\dots,\phi_l)$ where $p\in\NC(m,n)$, and if $p = \{B_1,\dots,B_l\}$, then
\[
 \phi_i \in \CC_k\left(\bigotimes_{j\in U_{B_i}} \alpha_j,\bigotimes_{j'\in L_{B_i}} \beta_{j'}\right).
\]

Graphically, we represent an object $\gamma = (m;\alpha_1,\dots,\alpha_m)$ by drawing $m$ ellipses, with the $i$th ellipse containing the graphical representation of $\alpha_i$. For example, consider the object $\gamma = (2;(3;1,2,0),(1;2)) \in \CC_3$. This can be represented as the following set of nested ellipses:

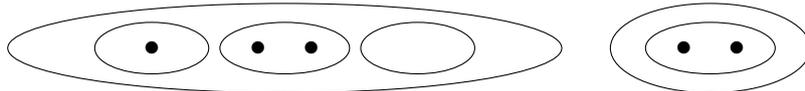
\begin{center}
\begin{tikzpicture}[every node/.style={inner sep =2pt}]
\def\dx{0.7}
\def\dy{2.5}

\node (l1t1) at (0,0) {$\bullet$};
\node (l1t2) at (2*\dx,0) {$\bullet$};
\node (l1t3) at (3*\dx,0) {$\bullet$};
\node (l1tphantom) at (5*\dx,0) {\textcolor{white}{$\bullet$}};
\node (l1t4) at (10*\dx,0) {$\bullet$};
\node (l1t5) at (11*\dx,0) {$\bullet$};
\node [fit=(l1t1)] (l2t1) [tube] {};
\node [fit=(l1t2)(l1t3)] (l2t2) [tube] {};
\node [fit=(l1tphantom)] (l2t3) [tube] {};
\node [fit=(l1t4)(l1t5)] (l2t4) [tube] {};
\node [fit=(l2t1)(l2t2)(l2t3)] (l3t1) [tube] {};
\node [fit=(l2t4)] (l3t1) [tube] {};

\end{tikzpicture}
\end{center}
To represent a morphism, we treat each of the ellipses as the opening to a tube that represents a corresponding partition defining the morphism. For example, consider a morphism $\rho \in \CC_3$. 

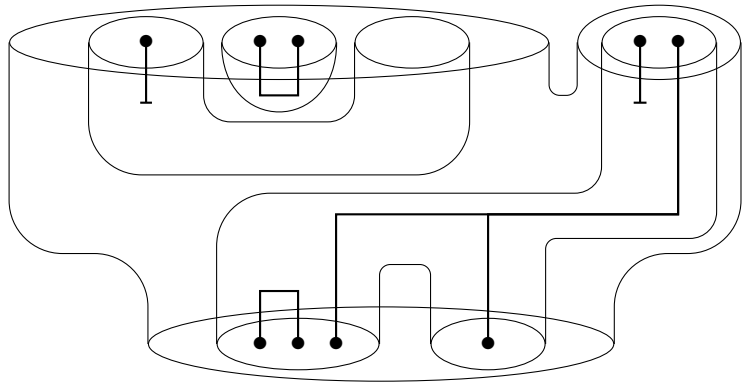
\begin{center}
\begin{tikzpicture}[every node/.style={inner sep =2pt}]
\def\dx{0.5}
\def\dy{4}
\node (l1t1) at (0,0) {$\bullet$};
\node (l1t2) at (3*\dx,0) {$\bullet$};
\node (l1t3) at (4*\dx,0) {$\bullet$};
\node (l1tphantom) at (7*\dx,0) {\textcolor{white}{$\bullet$}};
\node (l1t4) at (13*\dx,0) {$\bullet$};
\node (l1t5) at (14*\dx,0) {$\bullet$};
\node (l1b1) at (3*\dx,-\dy) {$\bullet$};
\node (l1b2) at ($(l1b1.center)+(\dx,0)$) {$\bullet$};
\node (l1b3) at ($(l1b2.center)+(\dx,0)$) {$\bullet$};
\node (l1b4) at ($(l1b3.center)+(4*\dx,0)$) {$\bullet$};
\node (terminal1) at ($(l1t1.center)+(0,-0.2*\dy)$) {};
\draw [thick] (l1t1.center) -- (terminal1.center);
\draw [thick] (terminal1.west) -- (terminal1.east);
\node [thick] (terminal2) at ($(l1t4.center)+(0,-0.2*\dy)$) {};
\draw [thick] (l1t4.center) -- (terminal2.center);
\draw [thick] (terminal2.west) -- (terminal2.east);
\draw [thick] (l1t2.center) r-du[du distance =20] (l1t3.center);
\draw [thick] (l1b1.center) r-ud[ud distance =20] (l1b2.center);
\draw [thick] (l1t5.center) |-|[ratio=0.57] (l1b3.center);
\draw [thick] (l1t5.center) |-|[ratio=0.57] (l1b4.center);

\node [fit=(l1t1)] (l2t1) [tube] {};
\node [fit=(l1t2)(l1t3)] (l2t2) [tube] {};
\node [fit=(l1tphantom)] (l2t3) [tube] {};
\node [fit=(l1t4)(l1t5)] (l2t4) [tube] {};
\node [fit=(l1b1)(l1b2)(l1b3)] (l2b1) [tube] {};
\node [fit=(l1b4)] (l2b2) [tube] {};
\draw[rounded corners=20] (l2t1.west) r-du[du distance=50] (l2t3.east);
\draw[rounded corners=10] (l2t1.east) r-du[du distance=30] (l2t3.west);
\draw (l2t2.west) to[out=-90,in=-90,distance=35] (l2t2.east);
\draw[rounded corners] (l2b1.east) r-ud[ud distance=30] (l2b2.west);
\draw[rounded corners=10] (l2t4.west) |- ($(l2t4.center)!0.5!(l2b1.center)$) [rounded corners=20] -| (l2b1.west) ;
\draw[rounded corners=10] (l2t4.east) |- ($(l2t4.east)!0.65!(l2b2.east)+(0.3*\dx,0)$) [rounded corners] -| (l2b2.east) ;

\node [fit=(l2t1)(l2t2)(l2t3)] (l3t1) [tube,inner sep=0pt] {};
\node [fit=(l2t4)] (l3t2) [tube,inner sep=0pt] {};
\node [fit=(l2b1)(l2b2)] (l3b1) [tube,inner sep=0pt] {};
\draw[rounded corners=20] (l3t1.west) |-|[ratio=0.7] (l3b1.west);
\draw[rounded corners=20] (l3t2.east) |-|[ratio=0.7] (l3b1.east);
\draw[rounded corners] (l3t1.east) r-du[du distance=20] (l3t2.west);


\end{tikzpicture}
\end{center}
%

This morphism should be thought of as a nesting of three levels of non-crossing partitions, with each level fully contained in the one above. The domain of $\rho$ is $\gamma$ from above, and the codomain is $\delta = (1;(2;3,1))$. We can write $\rho = (\shirt;\phi_1)$, where $\phi_1 \in \CC_2((4;1,2,0,2),(2;3,1))$ can be written as $(q;(r_1,r_2,r_3))$ with 
\[
 q=\tricky \otimes \pants\text{ and }r_1 = \cuppart ,\,r_2 = \paircup,\,r_3 = \cuppart\otimes\paircap\otimes\pants.
 \]
(Notice that $r_1, r_2$ and $r_3$ are completely contained in the blocks of $q$.) In other words, we have 
\[
\rho= \big(\shirt;\big(\tricky \otimes \pants; \cuppart , \paircup  ,  \cuppart\otimes \paircap  \otimes \pants\big)\big).
\]

Composition in $\CC_k$ is defined analogously to composition in $\NC$, where morphisms are stacked vertically, and the graphical representation of the middle objects are identified. Then any block in the vertical stack with no connection to the upper or lower openings is deleted. 

\begin{example}
The following diagram demonstrates the composition of morphisms $\rho \in \CC_2((3;2,3,1),(2;2,1))$ and $\pi \in \CC_2((1;1),(3;2,3,1))$.
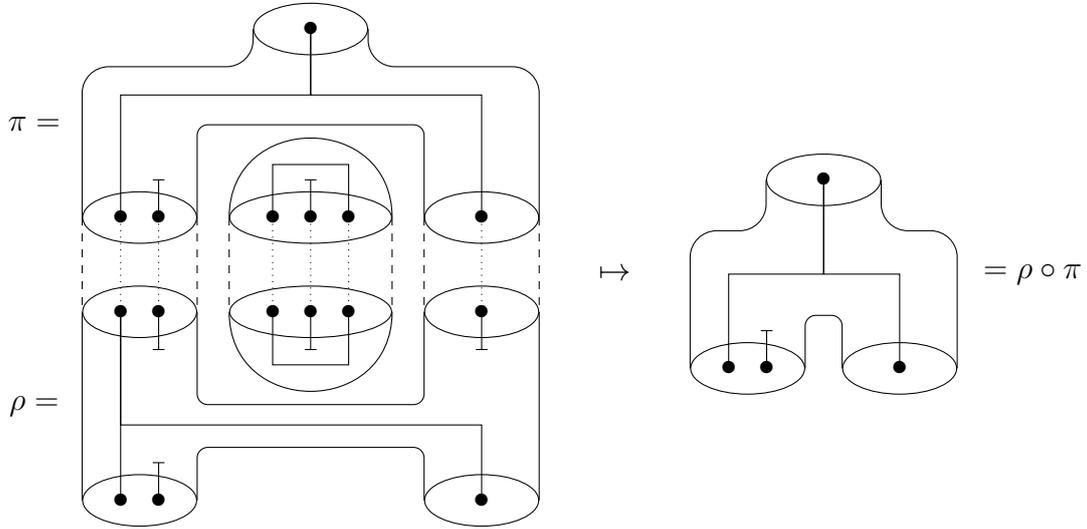
\begin{center}
\scalebox{1}{
\begin{tikzpicture}[every node/.style={inner sep =2pt}]
\def\dx{0.5}
\def\dy{2.5}
\begin{scope}

\node (pi) at (-2*\dx,-0.5*\dy) {$\pi = \ \ $};
\node (rho) at (-2*\dx,-2*\dy) {$\rho = \ \ $};

\node (l1t1) at (5*\dx,0) {$\bullet$};
\node [fit=(l1t1)] (l2t1) [tube] {};

\node (l1tm1) at (0,-\dy) {$\bullet$};
\node (l1tm2) at (\dx,-\dy) {$\bullet$};
\node (l1tm3) at (4*\dx,-\dy) {$\bullet$};
\node (l1tm4) at (5*\dx,-\dy) {$\bullet$};
\node (l1tm5) at (6*\dx,-\dy) {$\bullet$};
\node (l1tm6) at (9.5*\dx,-\dy) {$\bullet$};
\node [fit=(l1tm1)(l1tm2)] (l2tm1) [tube] {};
\node [fit=(l1tm3)(l1tm4)(l1tm5)] (l2tm2) [tube] {};
\node [fit=(l1tm6)] (l2tm3) [tube] {};

\node (l1bm1) at (0,-1.5*\dy) {$\bullet$};
\node (l1bm2) at (\dx,-1.5*\dy) {$\bullet$};
\node (l1bm3) at (4*\dx,-1.5*\dy) {$\bullet$};
\node (l1bm4) at (5*\dx,-1.5*\dy) {$\bullet$};
\node (l1bm5) at (6*\dx,-1.5*\dy) {$\bullet$};
\node (l1bm6) at (9.5*\dx,-1.5*\dy) {$\bullet$};
\node [fit=(l1bm1)(l1bm2)] (l2bm1) [tube] {};
\node [fit=(l1bm3)(l1bm4)(l1bm5)] (l2bm2) [tube] {};
\node [fit=(l1bm6)] (l2bm3) [tube] {};

\node (l1b1) at (0,-2.5*\dy) {$\bullet$};
\node (l1b2) at (\dx,-2.5*\dy) {$\bullet$};
\node (l1b3) at (9.5*\dx,-2.5*\dy) {$\bullet$};
\node [fit=(l1b1)(l1b2)] (l2b1) [tube] {};
\node [fit=(l1b3)] (l2b2) [tube] {};

\draw[rounded corners=10] (l2t1.west) |-|[ratio=0.2] (l2tm1.west);
\draw[rounded corners=10] (l2t1.east) |-|[ratio=0.2] (l2tm3.east);
\draw[rounded corners] (l2tm1.east) r-ud[ud distance=35] (l2tm3.west);
\draw (l2tm2.west) to[out=90,in=90,distance=40] (l2tm2.east);
\draw (l2bm1.west) -- (l2b1.west);
\draw[rounded corners] (l2bm1.east) r-du[du distance=35] (l2bm3.west);
\draw[rounded corners] (l2b1.east) r-ud[ud distance=20] (l2b2.west);
\draw (l2bm3.east) -- (l2b2.east);
\draw (l2bm2.west) to[out=-90,in=-90,distance=40] (l2bm2.east);

\draw (l1tm3.center) r-ud[ud distance=20] (l1tm5.center);
\node (x) at ([shift={(0,0.2*\dy)}]l1tm4) {};
\draw (l1tm4.center) -- (x.center);
\draw (x.west) -- (x.east);
\draw (l1bm3.center) r-du[du distance=20] (l1bm5.center);
\node (y) at ([shift={(0,-0.2*\dy)}]l1bm4) {};
\draw (l1bm4.center) -- (y.center);
\draw (y.west) -- (y.east);
\draw (l1t1.center) |-|[ratio=0.35] (l1tm1.center);
\draw (l1t1.center) |-|[ratio=0.35] (l1tm6.center);
\draw (l1bm1.center) -- (l1b1.center);
\draw (l1bm1.center) |-|[ratio=0.6] (l1b3.center);
\node (t1) at ([shift={(0,0.2*\dy)}]l1tm2) {};
\node (t2) at ([shift={(0,-0.2*\dy)}]l1bm2) {};
\node (t3) at ([shift={(0,-0.2*\dy)}]l1bm6) {};
\node (t4) at ([shift={(0,0.2*\dy)}]l1b2) {};
\draw (t1.west) -- (t1.east);
\draw (l1tm2.center) -- (t1.center);
\draw (t2.west) -- (t2.east);
\draw (l1bm2.center) -- (t2.center);
\draw (t3.west) -- (t3.east);
\draw (l1bm6.center) -- (t3.center);
\draw (t4.west) -- (t4.east);
\draw (l1b2.center) -- (t4.center);

\draw[dashed] (l2tm1.west) -- (l2bm1.west);
\draw[dashed] (l2tm1.east) -- (l2bm1.east);
\draw[dashed] (l2tm2.west) -- (l2bm2.west);
\draw[dashed] (l2tm2.east) -- (l2bm2.east);
\draw[dashed] (l2tm3.west) -- (l2bm3.west);
\draw[dashed] (l2tm3.east) -- (l2bm3.east);

\draw[dotted] (l1tm1.center) -- (l1bm1.center);
\draw[dotted] (l1tm2.center) -- (l1bm2.center);
\draw[dotted] (l1tm3.center) -- (l1bm3.center);
\draw[dotted] (l1tm4.center) -- (l1bm4.center);
\draw[dotted] (l1tm5.center) -- (l1bm5.center);
\draw[dotted] (l1tm6.center) -- (l1bm6.center);
\end{scope}

\begin{scope}[xshift=8cm,yshift=-2cm]
\node (rhocircpi) at (8*\dx,-0.5*\dy) {$= \rho\circ\pi$};

\node (l1t1) at (2.5*\dx,0) {$\bullet$};
\node [fit=(l1t1)] (l2t1) [tube] {};

\node (arrow) at (-3*\dx, -0.5*\dy) {$\mapsto$};

\node (l1b1) at (0,-\dy) {$\bullet$};
\node (l1b2) at (\dx,-\dy) {$\bullet$};
\node (l1b3) at (4.5*\dx,-\dy) {$\bullet$};
\node [fit=(l1b1)(l1b2)] (l2b1) [tube] {};
\node [fit=(l1b3)] (l2b2) [tube] {};

\draw[rounded corners=10] (l2t1.west) |-|[ratio=0.27] (l2b1.west);
\draw[rounded corners=10] (l2t1.east) |-|[ratio=0.27] (l2b2.east);
\draw[rounded corners] (l2b1.east) r-ud[ud distance=20] (l2b2.west);
\draw (l1t1.center) |-|[ratio=0.5] (l1b1.center);
\draw (l1t1.center) |-|[ratio=0.5] (l1b3.center);
\node (t1) at ([shift={(0,0.2*\dy)}]l1b2) {};
\draw (t1.west) -- (t1.east);
\draw (l1b2.center) -- (t1.center);
\end{scope}
\end{tikzpicture}

}
\end{center}
The resulting diagram represents the morphism $\rho\circ\pi\in\CC_2((1;1),(2;2,1))$. 
\end{example}

\begin{notation}\label{note: deletedblocks}
	We denote the number of distinct blocks deleted in a composition $\rho\circ\pi$ by $rb(\rho,\pi)$. For example, in the composition illustrated in the example above, there are four removed blocks---three from the inner partition and one from the outer partition---and so we have $rb(\rho,\pi) = 4$. 
\end{notation}

Finally, as in $\NC$, any morphism $\rho\in\Hom(\CC_k)$ has an adjoint $\rho^*$ obtained by vertical reflection. The visual representation of $\rho^*$ is of course then just the visual representation of $\rho$ upside down.

\subsubsection{Monoidal structure}

Given objects $\gamma = (m;\alpha_1,\dots,\alpha_m)$ and $\delta=(n;\beta_1,\dots,\beta_n)$ in $\CC_k$, the tensor product is given by
\[
\gamma \otimes \delta = (m+n;\alpha_1,\dots,\alpha_m,\beta_1,\dots,\beta_n).
\]
If $\rho = (p;\phi_1,\dots,\phi_l)$ and $\pi = (q;\psi_1,\dots,\psi_{l'})$ are morphisms in $\CC_k$ the tensor product is given by
\[
 \rho\otimes\pi = (p\otimes q;\phi_1,\dots,\phi_l,\psi_1,\dots,\psi_{l'}).
\]
Graphically, the tensor product on both objects and morphisms is just horizontal juxtaposition, analogous to the case for $\NC$.

\subsection{The category of tubular partitions}

\begin{definition}\label{def: functorsIk}
	For each $k$ we inductively define functors $I_k : \CC_k \to \CC_{k+1}$. Given $m\in\CC_1$ we define $I_1(m) := (m;0,\dots,0)\in\CC_2$, and given a morphism $p\in\CC_1(m,n)$ we define $I_1(p) := (p;\id_0,\dots,\id_0) \in \CC_2(I_1(m),I_1(n))$. For $k\geq 2$ and $\gamma = (m;\alpha_1,\dots,\alpha_m) \in \CC_k$ we define
\[
 I_k(\gamma) := (m;I_{k-1}(\alpha_1),\dots,I_{k-1}(\alpha_m)) \in \CC_{k+1},
\]
and given a morphism $\rho = (p;\phi_1,\dots,\phi_l) \in \CC_k(\gamma,\delta)$ we define
\[
 I_k(\rho) := (p;I_{k-1}(\phi_1),\dots,I_{k-1}(\phi_l)) \in \CC_{k+1}(I_k(\gamma),I_k(\delta)).
\]
\end{definition}

\begin{definition}
The category of tubular partitions $\TT$ is the union
\[
 \TT = \bigcup_{k=1}^\infty \CC_k,
\]
where we identify each $\CC_k$ with its image $I_k(\CC_k) \subset \CC_{k+1}$. The monoidal structure is inherited from each subcategory $\CC_k$ and the unit object is given by the image of $0\in\CC_1$.
\end{definition}

A key property of $\TT$ is the existence of an endofunctor $\Psi$, which can be interpreted visually as taking a morphism and putting it inside a cylindrical tube. This visual interpretation is the inspiration for the name \emph{tubular partitions} given to morphisms in $\TT$. Mathematically, the existence of this endofunctor is the manifestation of the self-similarity of the compact quantum group $\A_X$ obtained after applying Tannaka--Krein duality (see also the discussion before Proposition~\ref{prop:psiobj}).

\begin{prop}
There is an endofunctor $\Psi \in \End(\TT)$ satisfying
\[
\Psi(\alpha) = (1;\alpha)
\]
for $\alpha\in \CC_k$, and
\[
 \Psi(\varphi) = (\idpart ;\varphi)
\]
for $\varphi\in \CC_k(\alpha,\beta)$.
\end{prop}

\begin{proof}

We begin by showing that $\Psi$ is well-defined. Fix $\alpha\in \CC_k$. Then
\[
 I_{k+1}(\Psi(\alpha)) = I_{k+1}((1;\alpha)) = (1;I_k(\alpha)) = \Psi(I_k(\alpha)).
\]
Likewise, if $\varphi\in\CC_k(\alpha,\beta)$ then
\[
 I_{k+1}(\Psi(\varphi)) = I_{k+1}((\idpart;\varphi)) = (\idpart;I_k(\varphi)) = \Psi(I_k(\varphi))
\]
and so $\Psi$ respects the identification of $\CC_k$ and $I_k(\CC_k)) \subset \CC_{k+1}$, and is hence well-defined on $\TT$.

We have
\[
 \Psi(\id_\alpha) = (\idpart;\id_\alpha) = (\id_1;\id_\alpha) = \id_{(1;\alpha)} = \id_{\Psi(\alpha)},
\]
and given composable morphisms $\varphi, \pi \in \CC_k$ we have
\begin{align*}
 \Psi(\varphi\circ\pi) &= (\idpart;\varphi\circ\pi) \\
 & = (\idpart\circ\idpart;\varphi\circ\pi) \\
 & = (\idpart;\varphi)\circ(\idpart;\pi) \\
 & = \Psi(\varphi)\circ\Psi(\pi)
\end{align*}
as required.
\end{proof}

We represent the functor $\Psi$ graphically as follows: given $\alpha \in \CC_k$ we draw $\Psi(\alpha)$ as
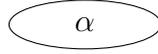
\begin{center}
\begin{tikzpicture}
\def\dy{1.5}
\node [tube,minimum width=2cm] at (0,\dy) {$\alpha$};
\end{tikzpicture}
\end{center}
and if $\varphi\in\CC_k(\alpha,\beta)$ then $\Psi(\phi)$ will be drawn as
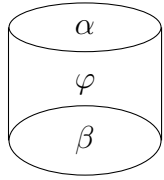
\begin{center}
	\begin{tikzpicture}
		\def\dy{1.5}
		\node (l2t1) [tube,minimum width=2cm] at (0,\dy) {$\alpha$};
		\node (l2b1) [tube,minimum width=2cm] at (0,0) {$\beta$};
		\node (phi) at (0,0.5*\dy) {$\varphi$};
		\draw (l2t1.west) to[out=-90,in=90] (l2b1.west);
		\draw (l2t1.east) to[out=-90,in=90] (l2b1.east);
	\end{tikzpicture}
\end{center}

Compare $\Psi$ with the category embedding functors $I_k$: for example, given the object $3\in\CC_1=\NC$ we have $I_1(3) = (3;0,0,0)$ which graphically looks like
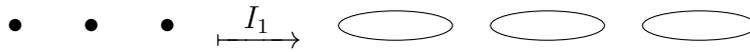
\begin{center}
\begin{tikzpicture}
\def\dx{1}
\node (t1) at (0,0) {$\bullet$};
\node (t2) at (\dx,0) {$\bullet$};
\node (t3) at (2*\dx,0) {$\bullet$};

\node (mapsto) at (3.2*\dx,0) {$\xmapsto{\mbox{ \ $I_1$ \ }}$};

\node[tube] (i1) at (5*\dx,0) {};
\node[tube] (i2) at (7*\dx,0) {};
\node[tube] (i3) at (9*\dx,0) {};
\end{tikzpicture}
\end{center}
while $\Psi(3) = (1;(3))$ which is graphically represented by
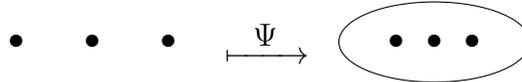
\begin{center}
\begin{tikzpicture}
\def\dx{1}
\node (t1) at (0,0) {$\bullet$};
\node (t2) at (\dx,0) {$\bullet$};
\node (t3) at (2*\dx,0) {$\bullet$};

\node (mapsto) at (3.3*\dx,0)  {$\xmapsto{\mbox{ \ $\Psi$ \ }}$};

\node (i1) at (5*\dx,0) {$\bullet$};
\node (i2) at (5.5*\dx,0) {$\bullet$};
\node (i3) at (6*\dx,0) {$\bullet$};
\node [fit=(i1)(i2)(i3)] (l1t1) [tube] {};
\end{tikzpicture}
\end{center}

\begin{remark}
The endofunctor $\Psi\in\End(\TT)$ is \textit{not} a monoidal functor. However, there is a canonical morphism $(\pants;\id_{\alpha\otimes\beta})$ from $\Psi(\alpha\otimes\beta)$ to $\Psi(\alpha)\otimes\Psi(\beta)$ for any $\alpha,\beta\in\TT$. This can be extended to any finite tensor product of objects in $\TT$.
\end{remark}

\begin{definition}
Given $\alpha_1, \dots, \alpha_n \in \CC_k$, define
\[
P_{\alpha_1,\dots,\alpha_n} := (p;(\id_{\alpha_1\otimes\dots\otimes\alpha_n})) \in \CC_k(\Psi(\alpha_1\otimes\dots\otimes\alpha_n),\Psi(\alpha_1)\otimes\dots\otimes\Psi(\alpha_n))
\]
where $p\in\NC(1,n)$ is the unique one-block partition on one upper and $n$ lower points.
\end{definition}
The morphism $P_{\alpha_1,\dots,\alpha_n}$ is represented by the diagram
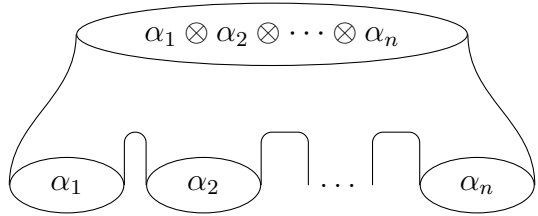
\begin{center}
\begin{tikzpicture}
\def\dx{1.8}
\def\dy{2}
\node (l2t1) [tube] at (1.5*\dx,\dy) {$\alpha_1\otimes\alpha_2\otimes\dots\otimes\alpha_n$};
\node (l2b1) [tube] at (0,0) {$\alpha_1$};
\node (l2b2) [tube] at (\dx,0) {$\alpha_2$};
\node (dots) at (2*\dx,0) {$\dots$};
\node (l2bn) [tube] at (3*\dx,0) {$\alpha_n$};
\draw (l2t1.west) to[out=-90,in=90] (l2b1.west);
\draw (l2t1.east) to[out=-90,in=90] (l2bn.east);
\draw[rounded corners] (l2b1.east) r-ud[ud distance=20] (l2b2.west);
\draw[rounded corners] (l2b2.east) r-ud[ud distance=20] (dots.west);
\draw[rounded corners] (dots.east) r-ud[ud distance=20] (l2bn.west);
\end{tikzpicture}
\end{center}

The following lemma will be needed in the proof of our main result Theorem~\ref{thm:TannakaKrein}. We start with some notation.

\begin{notation}\label{note:Psiof1}
	For each $n\in\mathbb{N}$ we write $\psi_n:=\Psi^n(1)$, where $1\in\NC$. 
\end{notation} 
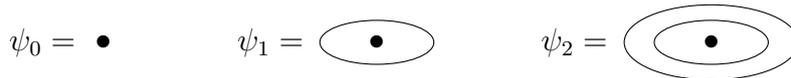
\begin{center}
\begin{tikzpicture}[every node/.style={inner sep =1.5pt}]
\def\dx{4}
\node (l11) at (0,0) {$\bullet$};
\node (l12) at (0.9*\dx,0) {$\bullet$};
\node (l13) at (2*\dx,0) {$\bullet$};
\node [fit=(l12)] (l21) [tube] {};
\node [fit=(l13)] (l22) [tube] {};
\node [fit=(l22)] (l31) [tube] {};

\node (psi0) at (-0.2*\dx,0) {$\psi_0=$};
\node (psi1) at (0.55*\dx,0) {$\psi_1=$};
\node (psi2) at (1.55*\dx,0) {$\psi_2=$};
\end{tikzpicture}
\end{center}

\begin{lemma}\label{lem:decomposingobjects}
Let $\alpha\in\mathcal{C}_k$. Then there exists $m\ge 1$, $\ell_1,\dots,\ell_m\in\{0,\dots,k-1\}$, and $S_\alpha\in\mathcal{T}(\alpha,\otimes_{j=1}^m\psi_{\ell_j})$ with $S_\alpha^*\circ S_\alpha=\id_\alpha$. 
\end{lemma} 

\begin{remark}\label{rem: picfordecomposingobjects}
	Before proving this lemma we illustrate the result with $\alpha=(1;(3;1,2,0)) $. The following figure represents a morphism $S\in\TT(\alpha,\psi_2\otimes\psi_2\otimes\psi_2\otimes\psi_1)$ satisfying $S^*\circ S=\id_\alpha$:
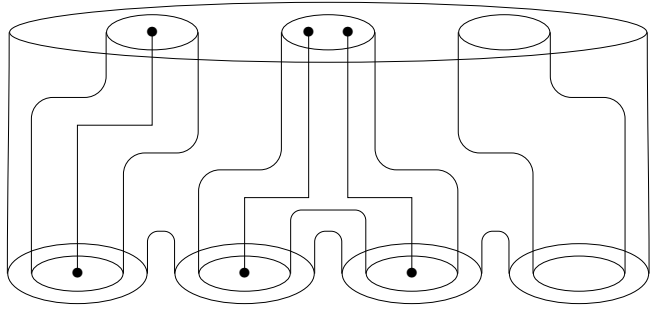
	\begin{center}
	\scalebox{0.8}{
		\begin{tikzpicture}[every node/.style={inner sep =1.5pt}]
			\def\dx{0.55}
			\def\tx{1.17}
			\def\dy{4}
			
			\node (l1t1) at (-4.5*\dx*\tx,0) {$\bullet$};
			\node (l1t2) at (-0.5*\dx*\tx,0) {$\bullet$};
			\node (l1t3) at (0.5*\dx*\tx,0) {$\bullet$};
			\node (l1tphantom) at (4.5*\dx*\tx,0) {\textcolor{white}{$\bullet$}};
			\node [fit=(l1t1)] (l2t1) [tube] {};
			\node [fit=(l1t2)(l1t3)] (l2t2) [tube] {};
			\node [fit=(l1tphantom)] (l2t3) [tube] {};
			\node [fit=(l2t1)(l2t2)(l2t3)] (l3t1) [tube] {};
			\node (l1b1) at (-7.5*\dx,-\dy) {$\bullet$};
			\node (l1b2) at (-2.5*\dx,-\dy) {$\bullet$};
			\node (l1b3) at (2.5*\dx,-\dy) {$\bullet$};
			\node (l1bphantom) at (7.5*\dx,-\dy) {\textcolor{white}{$\bullet$}};
			\node [fit=(l1b1)] (l2b1) [tube] {};
			\node [fit=(l1b2)] (l2b2) [tube] {};
			\node [fit=(l1b3)] (l2b3) [tube] {};
			\node [fit=(l1bphantom)] (l2b4) [tube] {};
			\node [fit=(l2b1)] (l3b1) [tube] {};
			\node [fit=(l2b2)] (l3b2) [tube] {};
			\node [fit=(l2b3)] (l3b3) [tube] {};
			\node [fit=(l2b4)] (l3b4) [tube] {};
			
			\draw[rounded corners=10] (l2t1.west) |-|[ratio=0.27] (l2b1.west);
			\draw[rounded corners=10] (l2t1.east) |-|[ratio=0.5] (l2b1.east);
			\draw[rounded corners=10] (l2t2.west) |-|[ratio=0.57] (l2b2.west);
			\draw[rounded corners=10] (l2t2.east) |-|[ratio=0.57] (l2b3.east);
			\draw[rounded corners=10] (l2t3.west) |-|[ratio=0.5] (l2b4.west);
			\draw[rounded corners=10] (l2t3.east) |-|[ratio=0.27] (l2b4.east);
			\draw[rounded corners=5] (l2b2.east) r-ud[ud distance=30] (l2b3.west);

			\draw (l3t1.west) to[out=-90,in=90] (l3b1.west);
			\draw (l3t1.east) to[out=-90,in=90] (l3b4.east);
			\draw[rounded corners=5] (l3b1.east) r-ud[ud distance=20] (l3b2.west);
			\draw[rounded corners=5] (l3b2.east) r-ud[ud distance=20] (l3b3.west);
			\draw[rounded corners=5] (l3b3.east) r-ud[ud distance=20] (l3b4.west);
			
			\draw (l1t1.center) |-|[ratio=0.385] (l1b1.center);
			\draw (l1t2.center) |-|[ratio=0.685] (l1b2.center);
			\draw (l1t3.center) |-|[ratio=0.685] (l1b3.center);

		\end{tikzpicture}}
	\end{center}
\end{remark}

\begin{proof}[Proof of Lemma~\ref{lem:decomposingobjects}]
	We prove this by induction on $k$. If $\alpha=n\in\mathcal{C}_1$, then $S_n:=\id_n\in\mathcal{T}(n,n)=\mathcal{T}(n,\psi_0^{\otimes n})$ satisfies $S_n^*\circ S_n=\id_n$. Assume the result holds for all $\alpha\in\mathcal{C}_k$, and let $\gamma:=(p;(\alpha_1,\dots,\alpha_p))\in\mathcal{C}_{k+1}$ with each $\alpha_i\in\mathcal{C}_k$. For each $1\le i\le p$ let $S_{\alpha_i}\in \mathcal{T}(\alpha_i,\otimes_{j=1}^{m_i}\psi_{\ell_{i,j}})$ where each $\ell_{i,j}\in\{0,\dots,k-1\}$. Then 
	\[
	S_\gamma:=\bigotimes_{i=1}^p\left(P_{\psi_{\ell_{i,1}},\dots,\psi_{\ell_{i,m_i}}}\circ\Psi(S_{\alpha_i})\right)
	\]
	satisfies $S_\gamma\in\TT(\gamma,\otimes_{i=1}^p\otimes_{j=1}^{m_i}\psi_{\ell_{i,j}})$, and 
	\begin{align*}
	S_{\gamma}^*\circ S_\gamma&=\left(\bigotimes_{i=1}^p\left(\Psi(S_{\alpha_i})^*\circ P_{\psi_{\ell_{i,1}},\dots,\psi_{\ell_{i,m_i}}}^*\right)\right)\circ \left(\bigotimes_{i=1}^p\left(P_{\psi_{\ell_{i,1}},\dots,\psi_{\ell_{i,m_i}}}\circ\Psi(S_{\alpha_i})\right)\right) \\
	&= \bigotimes_{i=1}^p\Psi(S_{\alpha_i}^*\circ S_{\alpha_i})\\
	&= \bigotimes_{i=1}^p \Psi(\id_{\alpha_i})\\
	&= \id_\gamma.
	\end{align*}
\end{proof}

\subsection{Rigidity} To see that the category $\TT$ is rigid, we argue by induction. We know $\CC_1 = \NC$ is rigid with self-dual objects; that is, $\overline{m}=m$ for all $m\in\NC$. Assume $\CC_k$ is rigid. Fix an object $\gamma = (m;\alpha_1,\dots,\alpha_m) \in \CC_{k+1}$. Then we define
\[
 \overline{\gamma} := (m;\overline{\alpha_m},\dots,\overline{\alpha_1}),
\]
which visually is the mirror image of $\gamma$. With morphisms given by
\[
 \eta_\gamma := (\eta_m;\eta_{\alpha_1},\dots,\eta_{\alpha_m}) \in \CC_{k+1}(0,\gamma\otimes\overline{\gamma}), \ \epsilon_\gamma := (\epsilon_m;\epsilon_{\alpha_m},\dots,\epsilon_{\alpha_1}) \in \CC_{k+1}(\overline{\gamma}\otimes\gamma,0),
\]
we have 
\[
(\id_\gamma\otimes\epsilon_\gamma)\circ (\eta_\gamma\otimes\id_\gamma)=\id_\gamma\text{ and }(\epsilon_\gamma\otimes\id_{\overline{\gamma}})\circ(\id_{\overline{\gamma}}\otimes \eta_\gamma)=\id_{\overline{\gamma}}.
\]
Hence $\CC_{k+1}$ is rigid. 

We illustrate the above identities with the following example, where $\gamma = \Psi(\alpha)$.

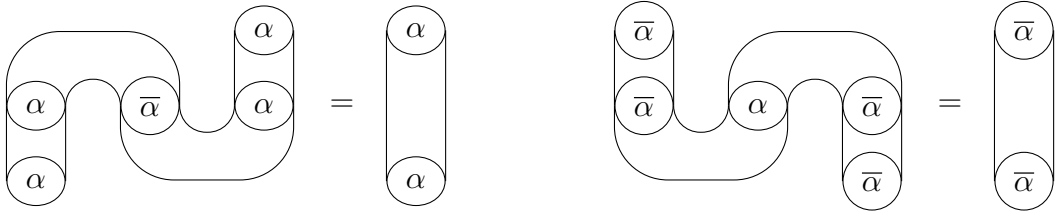
\begin{center}
	\begin{tikzpicture}
	\tikzset{tube/.style={draw, ellipse, minimum width=1pt}}
		\def\dy{1}
		\def\dx{1.5}
		\begin{scope}
		\node (l2t1) [tube] at (2*\dx,0) {$\alpha$};
		\node (l2m1) [tube] at (0,-\dy) {$\alpha$};
		\node (l2m2) [tube] at (\dx,-\dy) {$\overline{\alpha}$};
		\node (l2m3) [tube] at (2*\dx,-\dy) {$\alpha$};
		\node (l2b1) [tube] at (0,-2*\dy) {$\alpha$};
				
		\draw (l2t1.west) to[out=-90,in=90] (l2m3.west);
		\draw (l2t1.east) to[out=-90,in=90] (l2m3.east);
		\draw (l2m1.west) to[out=-90,in=90] (l2b1.west);
		\draw (l2m1.east) to[out=-90,in=90] (l2b1.east);
		\draw[rounded corners=20] (l2m1.west)  r-ud[ud distance=28](l2m2.east);
		\draw[rounded corners=10] (l2m1.east)  r-ud[ud distance=10](l2m2.west);
		\draw[rounded corners=20] (l2m2.west)  r-du[du distance=28](l2m3.east);
		\draw[rounded corners=10] (l2m2.east)  r-du[du distance=10](l2m3.west);
		\end{scope}
		
		\begin{scope}[xshift=5cm]
		\node (eq) at (-0.65*\dx,-\dy) {$=$};
		\node (l2t1) [tube] at (0,0) {$\alpha$};
		\node (l2b1) [tube] at (0,-2*\dy) {$\alpha$};
		\draw (l2t1.west) to[out=-90,in=90] (l2b1.west);
		\draw (l2t1.east) to[out=-90,in=90] (l2b1.east);
		\end{scope}
		
		\begin{scope}[xshift=8cm]
		\node (l2t1) [tube] at (0,0) {$\overline{\alpha}$};
		\node (l2m1) [tube] at (0,-\dy) {$\overline{\alpha}$};
		\node (l2m2) [tube] at (\dx,-\dy) {$\alpha$};
		\node (l2m3) [tube] at (2*\dx,-\dy) {$\overline{\alpha}$};
		\node (l2b1) [tube] at (2*\dx,-2*\dy) {$\overline{\alpha}$};
				
		\draw (l2t1.west) to[out=-90,in=90] (l2m1.west);
		\draw (l2t1.east) to[out=-90,in=90] (l2m1.east);
		\draw (l2m3.west) to[out=-90,in=90] (l2b1.west);
		\draw (l2m3.east) to[out=-90,in=90] (l2b1.east);
		\draw[rounded corners=20] (l2m1.west)  r-du[du distance=28](l2m2.east);
		\draw[rounded corners=10] (l2m1.east)  r-du[du distance=10](l2m2.west);
		\draw[rounded corners=20] (l2m2.west)  r-ud[ud distance=28](l2m3.east);
		\draw[rounded corners=10] (l2m2.east)  r-ud[ud distance=10](l2m3.west);
		\end{scope}
		
		\begin{scope}[xshift=13cm]
		\node (eq) at (-0.65*\dx,-\dy) {$=$};
		\node (l2t1) [tube] at (0,0) {$\overline{\alpha}$};
		\node (l2b1) [tube] at (0,-2*\dy) {$\overline{\alpha}$};
		\draw (l2t1.west) to[out=-90,in=90] (l2b1.west);
		\draw (l2t1.east) to[out=-90,in=90] (l2b1.east);
		\end{scope}
	\end{tikzpicture}
\end{center}

\subsection{Left and right rotation}

Let $\gamma = (m;\alpha_1,\dots,\alpha_m), \delta = (n;\beta_1,\dots,\beta_n) \in \TT$ and $\rho = (p;\phi_1,\dots,\phi_l) \in \TT(\gamma,\delta)$. Define the \emph{left rotation} of $\rho$ by
\[
 L(\rho) = (\id_{\Psi(\overline{\alpha_1})} \otimes \rho) \circ (\eta_{\Psi(\overline{\alpha_1})} \otimes \id_{\Psi(\alpha_2)}\otimes\dots\otimes\id_{\Psi(\alpha_m)})
\]
and \emph{right rotation} of $\rho$ by
\[
 R(\rho) = (\rho \otimes \id_{\Psi(\overline{\alpha_m})})\circ (\id_{\Psi(\alpha_1)}\otimes\dots\otimes\id_{\Psi(\alpha_{m-1})}\otimes  \eta_{\Psi(\alpha_m))})
\]
Visually, left (resp. right) rotations should be thought of as taking the top left (resp. right) outer-most ellipse and stretching and flipping the attached tube so that the opening is positioned at the bottom left (resp. right).

For example, if $\rho = (p;\varphi) \in \TT(\gamma,\delta)$, where $p \in \NC(m,n)$ consists of a single block, then left rotation of $\rho$ can be visualised as in the following diagram.
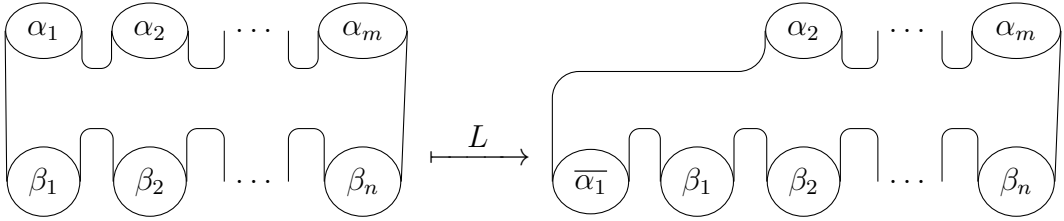
\begin{center}
\scalebox{1}{
\begin{tikzpicture}
\tikzset{tube/.style={draw, ellipse, minimum width=1pt}}
\begin{scope}
\def\dx{1.4}
\def\dy{1}

\node (l2t1) [tube] at (0,2) {$\alpha_1$};
\node (l2t2) [tube] at (\dx,2) {$\alpha_2$};
\node (tdots) at (2*\dx,2) {$\dots$};
\node (l2tm) [tube] at (3*\dx,2) {$\alpha_m$};
\node (l2b1) [tube] at (0,0) {$\beta_1$};
\node (l2b2) [tube] at (\dx,0) {$\beta_2$};
\node (bdots) at (2*\dx,0) {$\dots$};
\node (l2bn) [tube] at (3*\dx,0) {$\beta_n$};
\node (eq) at (4.1*\dx,0.5*\dy) {$\xmapsto{\mbox{ \ \ $L$ \ \ }}$};

\draw (l2t1.west) -- (l2b1.west);
\draw[rounded corners] (l2t1.east) r-du[ud distance=20] (l2t2.west);
\draw[rounded corners] (l2t2.east) r-du[ud distance=20] (tdots.west);
\draw[rounded corners] (tdots.east) r-du[ud distance=20] (l2tm.west);
\draw[rounded corners] (l2b1.east) r-ud[ud distance=20] (l2b2.west);
\draw[rounded corners] (l2b2.east) r-ud[ud distance=20] (bdots.west);
\draw[rounded corners] (bdots.east) r-ud[ud distance=20] (l2bn.west);
\draw (l2tm.east) -- (l2bn.east);
\end{scope}

\begin{scope}[xshift=8.6cm]
\def\dx{1.4}
\def\dy{1}

\node (l2t2) [tube] at (\dx,2) {$\alpha_2$};
\node (tdots) at (2*\dx,2) {$\dots$};
\node (l2tm) [tube] at (3*\dx,2) {$\alpha_m$};
\node (l2b0) [tube] at (-\dx,0) {$\overline{\alpha_1}$};
\node (l2b1) [tube] at (0,0) {$\beta_1$};
\node (l2b2) [tube] at (\dx,0) {$\beta_2$};
\node (bdots) at (2*\dx,0) {$\dots$};
\node (l2bn) [tube] at (3*\dx,0) {$\beta_n$};

\draw[rounded corners=10] (l2t2.west) |-|[ratio=0.27] (l2b0.west);
\draw[rounded corners] (l2b0.east) r-ud[ud distance=20] (l2b1.west);
\draw[rounded corners] (l2t2.east) r-du[ud distance=20] (tdots.west);
\draw[rounded corners] (tdots.east) r-du[ud distance=20] (l2tm.west);
\draw[rounded corners] (l2b1.east) r-ud[ud distance=20] (l2b2.west);
\draw[rounded corners] (l2b2.east) r-ud[ud distance=20] (bdots.west);
\draw[rounded corners] (bdots.east) r-ud[ud distance=20] (l2bn.west);
\draw (l2tm.east) -- (l2bn.east);
\end{scope}

\end{tikzpicture}
}
\end{center}

The flip explains why any object contained in the rotated opening is replaced by its dual object.

We can also rotate openings from bottom to top by $L(\rho^*)^*$ or $R(\rho^*)^*$. Visually, we are flipping the morphism vertically, left or right rotating, and then flipping vertically again. Note that any rotation can be reversed, since $L(L(\rho)^*)^* = \rho = R(R(\rho)^*)^*$.

\subsection{Generation of $\TT$} \label{subsec:generation}

In order to work with $\TT$ it will be useful to have a list of generators of each subcategory $\CC_k$. The category $\CC_1 = \NC$ is generated by $\{\idpart, \pants, \cappart\}$ and their adjoints, see for example \cite{BS2009}.

\begin{prop} For $k\geq 1$, $\CC_{k+1}$ is generated by
\begin{equation} \label{listofgenerators}
 \{\Psi(\varphi) : \varphi\in\Hom(\CC_k)\} \cup \{ P_{\alpha,\beta} : \alpha,\beta\in\CC_k \} \cup \{\cappart\},
\end{equation}
where we think of $\cappart$ as a morphism in $\CC_{k+1}$.
\end{prop}

\begin{proof} Fix $k\geq 1$. We begin by constructing the morphisms in $\CC_{k+1}$ necessary for left and right rotations; namely $\id_{\Psi(\alpha)}$ and $\eta_{\Psi(\alpha)}$ for $\alpha\in\CC_k$. We have $\id_{\Psi(\alpha)} = \Psi(\id_{\alpha})$ and $\id_\alpha\in\Hom(\CC_k)$. Then we have
\[
 \eta_{\Psi(\alpha)} = P_{\alpha,\overline{\alpha}} \circ \Psi(\eta_{\alpha}) \circ \cappart
\]
expressing $\eta_{\Psi(\alpha)}$ as a composition of morphisms in our generating set as required. This composition is shown in the following diagram:
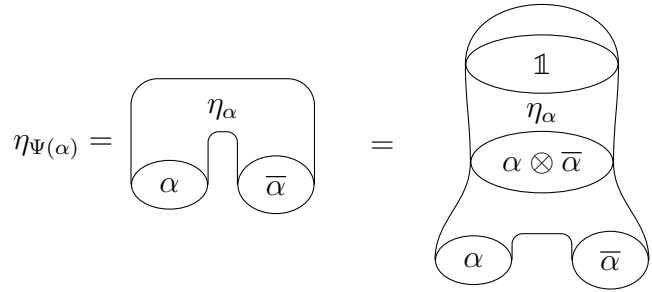
\begin{center}
\begin{tikzpicture}
\tikzset{tube/.style={draw, ellipse, minimum width=1cm}}
\begin{scope}
\def\dx{1.4}
\def\dy{1}
\node (etapsi) at (-\dx,0.5*\dy) {$\eta_{\Psi(\alpha)} = $};
\node (etaalpha) at (0.5*\dx,\dy) {$\eta_\alpha$};
\node (l2b1) [tube] at (0,0) {$\alpha$};
\node (l2b2) [tube] at (\dx,0) {$\overline{\alpha}$};
\node (eq) at (2*\dx,0.5*\dy) {$=$};
\draw[rounded corners=10] (l2b1.west) r-ud[ud distance=40] (l2b2.east); 
\draw[rounded corners] (l2b1.east) r-ud[ud distance=20] (l2b2.west);
\end{scope}

\begin{scope}[xshift=4cm,yshift=-1cm]
\def\dx{1.8}
\def\dy{1.3}

\node (l2t1) [tube,minimum width=2cm] at (0.5*\dx,2*\dy) {$\mathbb{1}$};
\node (etaalpha) at (0.5*\dx,1.5*\dy) {$\eta_\alpha$};
\node (l2m1) [tube] at (0.5*\dx,\dy) {$\alpha\otimes\overline{\alpha}$};
\node (l2b1) [tube] at (0,0) {$\alpha$};
\node (l2b2) [tube] at (\dx,0) {$\overline{\alpha}$};

\draw (l2t1.west) to[out=90,in=90,distance=30] (l2t1.east);
\draw (l2t1.west) to[out=-90,in=90] (l2m1.west);
\draw (l2t1.east) to[out=-90,in=90] (l2m1.east); 
\draw (l2m1.west) to[out=-90,in=90] (l2b1.west);
\draw (l2m1.east) to[out=-90,in=90] (l2b2.east);
\draw[rounded corners] (l2b1.east) r-ud[ud distance=10] (l2b2.west);
\end{scope}
\end{tikzpicture}
\end{center}

Now, through a series of left and right rotations, any morphism in $\CC_{k+1}$ can be transformed into a tensor product of finitely many morphisms of the form $\rho=(p;\varphi)$ where $\varphi\in\CC_k$ and $p\in\NC$ consists of a single block. Since rotations are reversible, it's enough to construct morphisms of this form using the generators. In other words, it is enough to construct a morphism $\rho$ as shown in the diagram below
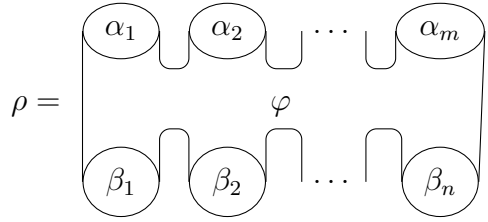
\begin{center}
\begin{tikzpicture}
\tikzset{tube/.style={draw, ellipse, minimum width=1cm}}
\def\dx{1.4}
\def\dy{1}
\node (rho) at (-0.8*\dx,\dy) {$\rho = $};
\node (phi) at (1.5*\dx,\dy) {$\varphi$};
\node (l2t1) [tube] at (0,2) {$\alpha_1$};
\node (l2t2) [tube] at (\dx,2) {$\alpha_2$};
\node (tdots) at (2*\dx,2) {$\dots$};
\node (l2tm) [tube] at (3*\dx,2) {$\alpha_m$};
\node (l2b1) [tube] at (0,0) {$\beta_1$};
\node (l2b2) [tube] at (\dx,0) {$\beta_2$};
\node (bdots) at (2*\dx,0) {$\dots$};
\node (l2bn) [tube] at (3*\dx,0) {$\beta_n$};

\draw (l2t1.west) -- (l2b1.west);
\draw[rounded corners] (l2t1.east) r-du[ud distance=20] (l2t2.west);
\draw[rounded corners] (l2t2.east) r-du[ud distance=20] (tdots.west);
\draw[rounded corners] (tdots.east) r-du[ud distance=20] (l2tm.west);
\draw[rounded corners] (l2b1.east) r-ud[ud distance=20] (l2b2.west);
\draw[rounded corners] (l2b2.east) r-ud[ud distance=20] (bdots.west);
\draw[rounded corners] (bdots.east) r-ud[ud distance=20] (l2bn.west);
\draw (l2tm.east) -- (l2bn.east);

\end{tikzpicture}
\end{center}
where $\varphi\in\CC_k(\alpha_1\otimes\dots\otimes\alpha_m,\beta_1\otimes\dots\otimes\beta_n)$. This follows from the equality
\[
 \rho = P_{\beta_1\otimes\dots\otimes\beta_n} \circ \Psi(\varphi) \circ P_{\alpha_1\otimes\dots\otimes\alpha_m}^*
\]
as illustrated below:
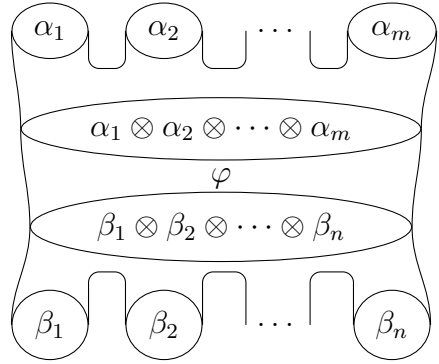
\begin{center}
\begin{tikzpicture}
\tikzset{tube/.style={draw, ellipse, minimum width=1cm}}
\def\dx{1.5}
\def\dy{1.3}

\node (l2t1) [tube] at (0,3*\dy) {$\alpha_1$};
\node (l2t2) [tube] at (\dx,3*\dy) {$\alpha_2$};
\node (tdots) at (2*\dx,3*\dy) {$\dots$};
\node (l2tm) [tube] at (3*\dx,3*\dy) {$\alpha_m$};
\node (phi) at (1.5*\dx,1.5*\dy) {$\varphi$};
\node (l2tm1) [tube] at (1.5*\dx,2*\dy) {$\alpha_1\otimes\alpha_2\otimes\dots\otimes\alpha_m$};
\node (l2bm1) [tube] at (1.5*\dx,\dy) {$\beta_1\otimes\beta_2\otimes\dots\otimes\beta_n$};
\node (l2b1) [tube] at (0,0) {$\beta_1$};
\node (l2b2) [tube] at (\dx,0) {$\beta_2$};
\node (bdots) at (2*\dx,0) {$\dots$};
\node (l2bn) [tube] at (3*\dx,0) {$\beta_n$};

\draw[rounded corners] (l2t1.east) r-du[ud distance=20] (l2t2.west);
\draw[rounded corners] (l2t2.east) r-du[ud distance=20] (tdots.west);
\draw[rounded corners] (tdots.east) r-du[ud distance=20] (l2tm.west);
\draw (l2t1.west) to[out=-90,in=90] (l2tm1.west);
\draw (l2tm.east) to[out=-90,in=90] (l2tm1.east);
\draw (l2tm1.west) to[out=-90,in=90] (l2bm1.west);
\draw (l2tm1.east) to[out=-90,in=90] (l2bm1.east);
\draw (l2bm1.west) to[out=-90,in=90] (l2b1.west);
\draw (l2bm1.east) to[out=-90,in=90] (l2bn.east);
\draw[rounded corners] (l2b1.east) r-ud[ud distance=20] (l2b2.west);
\draw[rounded corners] (l2b2.east) r-ud[ud distance=20] (bdots.west);
\draw[rounded corners] (bdots.east) r-ud[ud distance=20] (l2bn.west);
\end{tikzpicture}
\end{center}
\end{proof}

\subsection{$C^*$-tensor category}

To apply Tannaka-Krein duality we must first build a concrete rigid $C^*$-tensor category by associating Hilbert spaces and linear maps to the objects and morphisms of $\TT$. This construction is analogous to the definition of \emph{easy quantum groups} from partition categories, see \cite{BS2009}.

We begin by associating a Hilbert space $\HH_\alpha$ to each object $\alpha\in \TT$. We do this by defining an index set consisting of generalised words for an orthonormal basis of $\HH_\alpha$.

\subsubsection{Generalised words and Hilbert spaces}\label{sec:C*tensorcat} Fix a finite set $X$. Given an object $\alpha \in \TT$ we want to define a set $X^\alpha$, the elements of which we will refer to as \emph{generalised words of shape $\alpha$}. We do this inductively. If $m=0\in\CC_1$, then $X^m := \{\varnothing\}$. If $m \geq 1 \in \CC_1$ then
\[
 X^m := \{\mathbf{x} = x_1\cdots x_m : x_i \in X,\, 1\leq i \leq m\}
\]
is the usual set of words in $X$ of length $m$. For $u \in X^m$ we write $|u|:=m$. Then given an object $\gamma = (m;\alpha_1,\dots,\alpha_m) \in \CC_{k+1} \subset \TT$, inductively define
\[
 X^\gamma := \prod_{i=1}^m (X \times X^{\alpha_i}) = \{ v = (x_1,u_1)\cdots(x_m,u_m) : x_i\in X, \,u_i \in X^{\alpha_i}, \,1\leq i \leq m \}.
\]
We refer to an element of $X^\gamma$ as a \emph{generalised word of shape $\gamma$}, and for $v\in X^\gamma$ we write $|v| := \gamma$. Notice that there is a canonical bijection between $X^{\gamma\otimes\delta}$ and $X^\gamma \times X^\delta$. We now use these generalised words to associate Hilbert spaces to each object in $\TT$. 

\begin{definition}\label{def:Halpha}
If $\alpha \in \CC_k$ define the finite-dimensional Hilbert space $\HH_\alpha := \C^{X^\alpha}$ with canonical orthonormal spanning set $\{e_v : v\in X^\alpha\}$. Notice that $\HH_0 = \C$, and this assignment of Hilbert space respects the monoidal structure of the categories $\CC_k$ in the sense that $\HH_{\alpha\otimes\beta}$ is canonically isomorphic to $\HH_\alpha\otimes\HH_\beta$.
\end{definition}

\subsubsection{Linear maps}

We now want to associate a linear map $T_\rho : \HH_{\operatorname{dom}(\rho)} \to \HH_{\operatorname{cod}(\rho)}$ to each morphism $\rho$ in $\TT$. We do this with the aid of a function $\delta_\rho:X^{\operatorname{dom}(\rho)}\times X^{\operatorname{cod}(\rho)}\to \{0,1\}$, which we define inductively.  

Fix a morphism $p\in\NC(m,n)$. Let $\mathbf{x}=x_1\dots x_m \in X^m$ and $\mathbf{y}=y_1 \dots y_n \in X^n$ and label the top $m$ and bottom $n$ points in $p$ with $\mathbf{x}$ and $\mathbf{y}$, respectively, where the top and bottom points are both labelled left to right (in contrast to the labelling convention given in section \ref{subsec:catstructure}). Then define $\delta_p$ by $\delta_p(\mathbf{x},\mathbf{y}):=1$ if for every block in $p$ all points carry the same label, and otherwise $\delta_p(\mathbf{x},\mathbf{y}):=0$. 

We now assume that $\delta_\phi$ is defined for any morphism $\phi\in\Hom(\CC_k)$ for some $k\geq 1$. Let $\gamma=(m;\alpha_1,\dots,\alpha_m), \zeta=(n;\beta_1,\dots,\beta_n) \in \CC_{k+1}$ and fix $\rho = (p;\phi_1,\dots,\phi_l) \in \CC_{k+1}(\gamma,\zeta)$. Let $v = (x_1,u_1)\cdots(x_m,u_m)$ and $w=(y_1,z_1)\cdots(y_n,z_n)$ be generalised words of shape $\gamma$ and $\zeta$, respectively. Write $\mathbf{x} = x_1\cdots x_m, \mathbf{y} = y_1\cdots y_n$, and consider the objects
\[
\alpha' = \bigotimes_{i=1}^m \alpha_i , \ \beta' = \bigotimes_{i=1}^n \beta_i \in \CC_k,
\]
and the morphism
\[
\phi' = \bigotimes_{i=1}^l \phi_i \in \CC_k(\alpha',\beta').
\]
Effectively we are ignoring the outer-most blocks of the morphism $\rho \in \CC_{k+1}$ and treating what is left as a morphism in $\CC_k$. The concatenations $\mathbf{u} = u_1\cdots u_m$ and $\mathbf{z} = z_1\cdots z_n$ can now be identified with generalised words of shapes $\alpha'$ and $\beta'$, respectively. We can now define
\[
\delta_\rho(v,w) := \delta_p(\mathbf{x},\mathbf{y}) \delta_{\phi'}(\mathbf{u},\mathbf{z}).
\]

\begin{example}
	Let $\alpha = (1;2) \in \CC_2$ and suppose $\rho\in\CC_2(\alpha,\alpha)$ is the morphism depicted below.
	
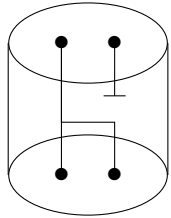
\begin{center}
		\begin{tikzpicture}[scale=0.7]
			\def\dx{1}
			\def\dy{2.5}
			\node (l1t1) at (0,0) {$\bullet$};
			\node (l1t2) at (\dx,0) {$\bullet$};
			\node [fit=(l1t1)(l1t2)] (l2t1) [tube] {};
			
			\node (l1b1) at (0,-\dy) {$\bullet$};
			\node (l1b2) at (\dx,-\dy) {$\bullet$};
			\node [fit=(l1b1)(l1b2)] (l2b1) [tube] {};
			\node (x) at ([shift={(0,-0.4*\dy)}]l1t2) {};
			
			\draw (l2t1.west) to (l2b1.west);
			\draw (l2t1.east) to (l2b1.east);
			\draw (l1t1.center) |-|[ratio=0.6] (l1b1.center);
			\draw (l1t1.center) |-|[ratio=0.6] (l1b2.center);
			\draw (l1t2.center) to (x.center);
			\draw (x.west) to (x.east);
		\end{tikzpicture}
	\end{center}
A generalised word of shape $\alpha$ has the form $(x,y_1y_2)$ where $x\in X$ and $y_1y_2 \in X^2$. We can label $\alpha$ by this generalised word by assigning the $x$ to the outer ellipse and $y_1, y_2$ to the two inner points.

Applying the definition of $\delta_\rho$ would then say that for example, given distinct $x,y\in X$
\[
\delta_\rho((x,xx),(y,xx)) = 0
\]
because the outer ellipses are connected but carry different labels. Likewise,
\[
\delta_\rho((x,xy),(x,xy)) = 0
\]
because the inner block containing three points would have one point labelled differently to the other two (the bottom right in this case). However,
\[
\delta_\rho((x,yx),(x,yy)) = 1
\]
since all points and ellipses connnected in a block carry the same label.
\end{example}

We are now ready to define the linear maps. 

\begin{definition}\label{def:LinearMaps}
	For a morphism $\rho$ in $\TT$ we define $T_{\rho} : \HH_{\operatorname{dom}(\rho)} \to \HH_{\operatorname{cod}(\rho)}$ by, for each $v\in X^{\operatorname{dom}(\rho)}$,
	\[
	T_\rho(e_v) := \sum_{\substack{w\in X^{\operatorname{cod}(\rho)} \\ \delta_\rho(v,w)=1}} e_w.
	\]
\end{definition}

\begin{lemma}\label{lem: propsofTs}
The assignment of a linear map $T_\rho$ to a tubular partition $\rho \in \TT$ satisfies
\begin{enumerate}
 \item \label{linearmaptensors} $T_{\rho\otimes\pi} = T_\rho\otimes T_\pi$,
 \item \label{linearmapcomposition} $T_\rho\circ T_\pi = |X|^{rb(\rho,\pi)} T_{\rho\circ\pi}$, and
 \item \label{linearmapadjoint} $(T_{\rho^*}) = T_\rho^*$.
\end{enumerate}
\end{lemma}

\begin{proof}
We have $X^{\alpha\otimes\beta} = X^\alpha \times X^\beta$ for $\alpha,\beta\in \TT$, and then (\ref{linearmaptensors}) follows by verifying that given $v\in X^{\dom(\rho)}, w\in X^{\dom(\pi)}$ and $v'\in X^{\cod(\rho)}, w'\in X^{\cod(\pi)}$ we have $\delta_{\rho\otimes\pi}((v,w),(v',w')) = \delta_\rho(v,v') \delta_\pi(w,w')$.

For (\ref{linearmapcomposition}), fix $v\in X^{\dom(\pi)}$. Then
\begin{align*}
 T_{\rho\circ\pi}(e_v) = \sum_{\delta_{\rho\circ\pi}(v,z)=1} e_z.
\end{align*}
Conversely,
\[
 T_\rho\circ T_\pi(e_v) = \sum_{\delta_\pi(v,w)=1} T_\rho(e_w)= \sum_{\substack{\delta_\pi(v,w)=1 \\ \delta_\rho(w,z)=1}} e_z.
\]
We need to find the number of possible generalised words $w \in X^{\cod(\pi)} = X^{\dom(\rho)}$ that are appearing in non-zero terms of the sum for a given $z\in X^{\cod(\rho)}$. Once $z$ is fixed, any block removed from the composition $\rho\circ\pi$ can have all elements labelled by any fixed letter $x\in X$. Hence there are $|X|^{rb(\rho,\pi)}$ such $w$, and the desired formula follows.

Finally, (\ref{linearmapadjoint}) follows from the fact that $\delta_\rho(v,w) = \delta_{\rho^*}(w,v)$ for any $v\in X^{\dom(\rho)}$ and $w \in X^{\cod(\rho)}$.
\end{proof}


\section{Tannaka-Krein duality} \label{sec:tkduality}

We can now state our main result, which says that the compact quantum group naturally associated to the category of tubular partitions $\TT$ via Tannaka--Krein duality is the quantum automorphism group $(\A_X,\Delta)$. Recall from Definitions~\ref{def:Halpha} and \ref{def:LinearMaps} the Hilbert spaces $\HH_\alpha$ and linear maps $T_\rho$ we associate to the elements of $\TT$. 

%
%
%

\begin{thm} \label{thm:TannakaKrein}
	Let $X$ be a finite set, and $\TT$ the category of tubular partitions. Let $\RR_X = (R, \{\HH_r\}_{r\in R}, \Hom(r,s)_{r,s\in R} )$ be the smallest complete concrete rigid $C^*$-tensor category containing the Banach spaces
	\[
	\RR_X(\alpha,\beta) := \spn\{T_\rho: \rho\in\TT(\alpha,\beta)\}\subset \BB(\HH_\alpha,\HH_\beta),
	\]
	where $\alpha,\beta\in\TT$. Then there is a family $\{u^\alpha\}_{\alpha\in\TT}\subseteq \operatorname{Rep}(\A_X)$ such that $(\mathbb{A}_X,\{u^\alpha\}_{\alpha\in\TT})$ is the unique universal $\RR_X$-admissible pair associated to $(\RR_X,\TT)$.
\end{thm}


It is proved in \cite[Theorem~4.5]{BR2025} that the map $\Aut(X^*)\times X\to X\times\Aut(X^*)$ given by $(g,x)\mapsto (g(x),g|_x)$ has a quantum analogue in the form of a homomorphism $\psi : \C^X \otimes \A_X \to \A_X\otimes\C^X$ satisfying $(\Delta\otimes\id)\psi = (\id\otimes\psi)(\psi\otimes\id)(\id\otimes\Delta)$. We begin the proof of Theorem~\ref{thm:TannakaKrein} by showing that $\psi$ induces a map on $\Rep(\A_X)$. 

\begin{prop} \label{prop:psiobj}
	Let $u = (u_{i,j})$ be a finite-dimensional unitary representation of $\A_X$ on $\HH$. Then there is a unitary representation $\psi(u) = (\psi(u)_{(x,i),(y,j)}) \in \Rep(\A_X)$ on the Hilbert space $\C^X\otimes \HH$
	such that the matrix elements satisfy the relation
	\[
	\psi(e_x\otimes u_{i,j}) = \sum_{y\in X}\psi(u)_{(x,i),(y,j)} \otimes e_y.
	\]
	Moreover, if $v \in \Rep(\A_X)$ is a representation on $\KK$, and $T\in\Rep(\A_X)(u,v)$ is an intertwiner between $u$ and $v$, then $	\id_{\C^X}\otimes T \in \BB(\C^X\otimes\HH,\C^X\otimes\KK)$	is an intertwiner between $\psi(u)$ and $\psi(v)$.
\end{prop}

\begin{proof}
	We have
	\begin{align*}
		\sum_{y\in X} \Delta(\psi(u)_{(x,i),(y,j)}) \otimes e_y &= (\Delta\otimes\id)\psi(e_x\otimes u_{i,j}) \\
		&= (\id\otimes\psi)(\psi\otimes\id)(\id\otimes\Delta)(e_x\otimes u_{i,j}) \\
		&= \sum_{1\le k\le d_u}(\id\otimes\psi)(\psi\otimes\id)(e_x\otimes u_{i,k}\otimes u_{k,j})) \\
		&= \sum_{\substack{z\in X\\ 1\le k\le d_u}} (\id\otimes\psi)(\psi(u)_{(x,i),(z,k)}\otimes e_z \otimes u_{k,j}) \\
		&= \sum_{\substack{y,z\in X\\ 1\le k\le d_u}}\psi(u)_{(x,i),(z,k)} \otimes \psi(u)_{(z,k),(y,j)} \otimes e_y,
	\end{align*}
	and comparing tensor factors gives 
	\[
	\Delta(\psi(u)_{(x,i),(y,j)}) = \sum_{\substack{z\in X\\ 1\le k\le d_u}} \psi(u)_{(x,i),(z,k)} \otimes \psi(u)_{(z,k),(y,j)}.
	\]
	To see that $\psi(u)$ is a representation, it remains to show that $\psi(u)$ is unitary. First note that 
	\begin{align*}
		\sum_{x\in X}\psi(u^*)_{(y,i),(x,j)}\otimes e_x &= \psi(e_y\otimes (u^*)_{i,j})\\
		&=\psi(e_y\otimes u_{j,i})^*\\
		&=\sum_{x\in X}\psi(u)_{(y,j),(x,i)}^*\otimes e_x\\
		&=\sum_{x\in X}(\psi(u)^*)_{(x,i),(y,j)}\otimes e_x,
	\end{align*}
	and hence we have 
	\[
	(\psi(u)^*)_{(x,i),(y,j)}=\psi(u^*)_{(y,i),(x,j)}.
	\]
	Now it follows that for each $x,y\in X$ and $1\le i,j\le d_u$ we have
	\begin{align*}
		&\sum_{\substack{z\in X\\ 1\le k\le d_u}}\psi(u)_{(x,i),(z,k)}\psi(u)_{(z,k),(y,j)}^*\otimes e_z\\
		&\hspace{3cm}= \sum_{\substack{z\in X\\ 1\le k\le d_u}}\psi(u)_{(x,i),(z,k)}\psi(u^*)_{(y,k),(z,j)}\otimes e_z\\
		&\hspace{3cm}= \sum_{1\le k\le d_u}\Big(\sum_{z\in X}\psi(u)_{(x,i),(z,k)}\otimes e_z\Big)\Big(\sum_{z'\in X}\psi(u^*)_{(y,k),(z',j)}\otimes e_{z'}\Big)\\
		&\hspace{3cm}=\sum_{1\le k\le d_u}\psi(e_x\otimes u_{i,k})\psi(e_y\otimes u_{k,j}^*)\\
		&\hspace{3cm}=\psi\Big(e_xe_y\otimes\sum_{1\le k\le d_u}u_{i,k}u_{k,j}^*\Big)\\
		&\hspace{3cm}=\delta_{x,y}\delta_{i,j}\psi(e_x\otimes 1_{\mathbb{A}_X})\\
		&\hspace{3cm}=\delta_{x,y}\delta_{i,j} \sum_{z\in X}a_{x,z}\otimes e_z,
	\end{align*}
	and hence for each $z\in X$ we have 
	\[
	\sum_{1\le k\le d_u}\psi(u)_{(x,i),(z,k)}\psi(u)_{(z,k),(y,j)}^*=\delta_{x,y}\delta_{i,j} a_{x,z}.
	\]
	Hence 
	\[
	\sum_{\substack{z\in X\\ 1\le k\le d_u}}\psi(u)_{(x,i),(z,k)}\psi(u)_{(z,k),(y,j)}^*=\delta_{x,y}\delta_{i,j}\sum_{z\in X}a_{x,z}=\delta_{x,y}\delta_{i,j}1_{\mathbb{A}_X}
	\]
	showing that $\psi(u)$ is a unitary representation. 
	
	For the second claim, we view $T=(T_{i,j})$ and $(\id_{\C^X}\otimes T)_{(x,i),(y,j)}$ as matrices; we need to show that for any $x,y\in X$ and $1\leq i \leq d_v$, $1\leq j \leq d_u$
	\[
	\sum_{\substack{z\in X\\ 1\le k\le d_u}}(\id_{\C^X}\otimes T)_{(x,i),(z,k)}\psi(u)_{(z,k),(y,j)}=\sum_{\substack{z\in X\\ 1\le k\le d_u}}\psi(v)_{(x,i),(z,k)}(\id_{\C^X}\otimes T)_{(z,k),(y,j)},
	\]
	which is equivalent to 
	\begin{equation}\label{eq:IntIdTesorT}
		\sum_{1\le k\le d_u}T_{i,k}\psi(u)_{(x,k),(y,j)}=\sum_{1\le k\le d_u}\psi(v)_{(x,i),(y,k)}T_{k,j}.
	\end{equation}
	Equation \ref{eq:IntIdTesorT} follows from the calculation:
	\begin{align*}
		\sum_{y\in X}\Big(\sum_{1\le k\le d_u}T_{i,k}\psi(u)_{(x,k),(y,j)}\Big)\otimes e_y 
		&= \sum_{1\le k\le d_u}T_{i,k}\Big(\sum_{y\in X}\psi(u)_{(x,k),(y,j)}\otimes e_y\Big)\\
		&= \sum_{1\le k\le d_u}T_{i,k}\psi(e_x\otimes u_{k,j})\\
		&=\psi\Big(e_x\otimes \Big(\sum_{1\le k\le d_u}T_{i,k}u_{k,j}\Big)\Big)\\
		&=\psi\Big(e_x\otimes \Big(\sum_{1\le k\le d_u}v_{i,k}T_{k,j}\Big)\Big)\\
		&=\sum_{y\in X}\Big(\sum_{1\le k\le d_u}\psi(v)_{(x,i),(y,k)}T_{k,j}\Big)\otimes e_y.
	\end{align*}
\end{proof}

\begin{prop} \label{prop:psirep}
	Fix finite-dimensional unitary representations $u, v$ of $\A_X$ on $\HH, \KK$ respectively. The linear mapping
	\[
	T: \C^X\otimes\HH\otimes\KK \to \C^X\otimes\HH\otimes\C^X\otimes\KK
	\]
	satisfying
	\[
	T(e_x\otimes h\otimes k) =e_x\otimes h\otimes e_x \otimes k
	\]
	for all $x\in X,h\in\HH,k\in\KK$, is an intertwiner between $\psi(u\otimes v)$ and $\psi(u)\otimes\psi(v)$.
\end{prop}

Before proving Proposition \ref{prop:psirep} we need a lemma. 

\begin{lemma} \label{lemma:orthogpsi}
	Fix $u,v \in\Rep(\A_X)$ and $x_1\neq x_2\in X$. Then
	\[
	\psi(u)_{(x_1,i),(y,j)}\psi(v)_{(x_2,i'),(y,j')} = 0
	\]
	for any $1\leq i,j \leq d_u, 1\leq i',j'\leq d_v$ and $y\in X$.
\end{lemma}

\begin{proof}
	We have 
	\begin{align*}
		\sum_{y\in X}\psi(u)_{(x,i),(y,j)}\otimes e_y &= \psi(e_x\otimes u_{i,j}) \\
		&= \psi(e_x\otimes u_{i,j})\psi(e_x\otimes 1_{\A_X}) \\
		&= \psi(e_x\otimes u_{i,j})\sum_{y'\in X}a_{x,y'}\otimes e_{y'}\\
		&= \sum_{y\in X}\psi(u)_{(x,i),(y,j)}a_{x,y}\otimes e_y,
	\end{align*}
	and it follows that $\psi(u)_{(x,i),(y,j)}=\psi(u)_{(x,i),(y,j)}a_{x,y}$. Similarly, we have $a_{x,y}\psi(u)_{(x,i),(y,j)}=\psi(u)_{(x,i),(y,j)}$. Hence we have 
	\[
	\psi(u)_{(x_1,i),(y,j)}\psi(v)_{(x_2,i'),(y,j')} = \psi(u)_{(x_1,i),(y,j)} a_{x_1,y} a_{x_2,y} \psi(v)_{(x_2,i'),(y,j')} = 0.
	\]
\end{proof}


\begin{proof}[Proof of Proposition~\ref{prop:psirep}]
	If we write $T$ as a matrix $T=(T_{(x,i,y,j),(z,k,l)})$, where $x,y,z\in X$, $1\le i,k\le d_u$, $1\le j,l\le d_v$,  we have
	\[
	T_{(x,i,y,j),(z,k,l)}=\begin{cases} 1 & \text{if $x=y=z$, $i=k$, $j=l$}\\ 0 &\text{otherwise.}\end{cases}
	\]
	We need to prove that 
	\[
	(T\psi(u\otimes v))_{(x,i,y,j),(z,k,l)}=((\psi(u)\otimes\psi(v))T)_{(x,i,y,j),(z,k,l)},
	\]
	which is equivalent to 
	\begin{equation}\label{eq:TensorEntries}
		\delta_{x,y}\psi(u\otimes v)_{(x,i,j),(z,k,l)}=\psi(u)_{(x,i),(z,k)}\psi(v)_{(y,j),(z,l)}.
	\end{equation}
	By Lemma~\ref{lemma:orthogpsi} we know the right-hand side of \eqref{eq:TensorEntries} is
	\[
	\delta_{x,y}\psi(u)_{(x,i),(z,k)}\psi(v)_{(x,j),(z,l)}.
	\]
	We also have  
	\begin{align*}
		\psi(e_x\otimes (u\otimes v)_{(i,j),(k,l)}) &= \psi(e_x\otimes u_{i,k}v_{j,l})\\
		&= \psi(e_x\otimes u_{i,k})\psi(e_x\otimes v_{j,l})\\
		&= \sum_{z\in X}\psi(u)_{(x,i),(z,k)}\psi(v)_{(x,j),(z,l)}\otimes e_z,
	\end{align*}
	and it follows that 
	\[
	\psi(u\otimes v)_{(x,i,j),(z,k,l)}=\psi(u)_{(x,i),(z,k)}\psi(v)_{(x,j),(z,l)}.
	\]
	Hence \eqref{eq:TensorEntries} holds.
\end{proof}


\begin{proof}[Proof of Theorem \ref{thm:TannakaKrein}]
We define the representations $u^\alpha$, $\alpha\in\TT$, so that $(\A_X,\{u^\alpha\}_{\alpha\in\TT})$ is a universal $\RR_X$-admissible pair. Theorem~\ref{thm:Wang} will then give uniqueness and complete the proof. 

We define the representations $u^\alpha$ inductively: for $m\in \CC_1$ define $u^0:=a^{(0)}=(1_{\A_X})$, and $u^m  := (a^{(1)})^{\otimes m}$ for $m\ge 1$. Then assuming we have defined unitary representations $u^\alpha$ for any $\alpha \in \CC_k$, if $\gamma = (m;\alpha_1,\dots,\alpha_m) \in \CC_{k+1}$, define $ u^{\gamma} := \bigotimes_{i=1}^m \psi(u^{\alpha_i})$, where $\psi(u^{\alpha_i})$ is the unitary representation defined in Proposition~\ref{prop:psiobj}. Since $a^{(1)}$ is a unitary, we know that each $u^m$ is a unitary, and then it follows from Proposition~\ref{prop:psiobj} that each $u^\gamma$ is a unitary. 

It is clear from this definition that $u^{\alpha\otimes\beta} = u^\alpha\otimes u^\beta$ for all $\alpha, \beta \in \TT$. We now need to show that $u^\alpha(T\otimes 1_B) = (T\otimes 1_B)u^\beta$ for all  $\alpha,\beta\in\TT$ and $T\in\RR_X(\alpha,\beta)$. It suffices to show that this identity holds for maps of the form $T_\rho$ for $\rho\in\TT(\alpha,\beta)$, and for this we only need to consider $\rho$ one of the generating morphisms from \eqref{listofgenerators}. That is, we need to show that 
\begin{itemize}
	\item[(1)] $(T_{\idpart}\otimes 1_{\A_X})u^1 = u^1(T_{\idpart}\otimes 1_{\A_X})$,
	\item[(2)] $(T_{\pants}\otimes 1_{\A_X})u^1 = u^2(T_{\pants}\otimes 1_{\A_X})$, and 
	\item[(3)] $(T_{\cappart}\otimes 1_{\A_X})u^0 = u^1(T_{\cappart}\otimes 1_{\A_X})$;
\end{itemize}
and for each $k\ge 1$, $\alpha,\beta\in \CC_k$ and $\varphi\in\CC_k(\alpha,\beta)$, 
\begin{itemize}
	\item[(4)] $(T_{P_{\alpha,\beta}}\otimes 1_{\A_X})u^{\Psi(\alpha\otimes\beta)}=u^{\Psi(\alpha)\otimes\Psi(\beta)}(T_{P_{\alpha,\beta}}\otimes 1_{\A_X})$, and 
	\item[(5)] $(T_\varphi\otimes 1_{\A_X})u^{\alpha}=u^{\beta}(T_\varphi\otimes 1_{\A_X})\implies (T_{\Psi(\varphi)}\otimes 1_{\A_X})u^{\Psi(\alpha)}=u^{\Psi(\beta)}(T_{\Psi(\varphi)}\otimes 1_{\A_X})$.
\end{itemize}

Identity (1) follows immediately because $T_{\idpart}$ is the identity map on $\C^X$. For (2) and (3) we treat the elements as matrices. To see that (2) holds, first note that for $x,y,z\in X$ we have $(T_{\pants}\otimes 1_{\A_X})_{(x,y),z}=\delta_{x,z}\delta_{y,z}1_{\A_X}$. Then we have
\begin{align*}
 &((T_{\pants}\otimes 1_{\A_X})u^1)_{(x,y),z} = (u^2(T_{\pants}\otimes 1_{\A_X}))_{(x,y),z} \\
 &\hspace{1.5cm}\iff ((T_{\pants}\otimes 1_{\A_X})a^{(1)})_{(x,y),z} = ((a^{(1)}\otimes a^{(1)})(T_{\pants}\otimes 1_{\A_X}))_{(x,y),z}\\
 &\hspace{1.5cm}\iff \sum_{z'\in X}(T_{\pants}\otimes 1_{\A_X})_{(x,y),z'}a_{z'z}=\sum_{x',y'\in X}a_{x,x'}a_{y,y'}(T_{\pants}\otimes 1_{\A_X}))_{(x',y'),z}\\
 &\hspace{1.5cm}\iff \delta_{x,y}a_{x,z}=a_{x,z}a_{y,z}.
\end{align*}
We know from Remark~\ref{rmk:consequencefromAX} (i) that $\delta_{x,y}a_{x,z}=a_{x,z}a_{y,z}$, and so (2) holds.

To see that (3) holds, let $x\in X$. Then $(T_{\cappart}\otimes 1_{\A_X})_{x,1}=1_{\A_X}$, and so we have
\begin{align*}
	((T_{\cappart}\otimes 1_{\A_X})u^0)_{x,1} &= (u^1(T_{\cappart}\otimes 1_{\A_X}))_{x,1}\\
	&\iff ((T_{\cappart}\otimes 1_{\A_X})(1_{\A_x}))_{x,1} = (a^{(1)}(T_{\cappart}\otimes 1_{\A_X}))_{x,1}\\
	&\iff (T_{\cappart}\otimes 1_{\A_X})_{x,1}(1_{\A_x})_{1,1}=\sum_{y\in X} a_{x,y}(T_{\cappart}\otimes 1_{\A_X})_{y,1}\\
	&\iff 1_{\A_X}=\sum_{y\in X} a_{x,y}.
\end{align*}
We know from Remark~\ref{rmk:consequencefromAX} (ii) that $\sum_{y\in X} a_{x,y}=1_{\A_X}$, and so (3) holds.

To see that (4) holds, fix $k\ge 1$ and $\alpha,\beta\in \CC_k$. We have 
\[
u^{\Psi(\alpha\otimes\beta)}=u^{(1;\alpha\otimes\beta)}=\psi(u^{\alpha\otimes\beta})=\psi(u^\alpha\otimes u^\beta),
\]
and 
\[
u^{\Psi(\alpha)\otimes \Psi(\beta)}=u^{\Psi(\alpha)}\otimes u^{\Psi(\beta)}=u^{(1;\alpha)}\otimes u^{(1;\beta)}=\psi(u^\alpha)\otimes \psi(u^\beta).
\]
So we need to show that $(T_{P_{\alpha,\beta}}\otimes 1_{\A_X})\psi(u^\alpha\otimes u^\beta)=(\psi(u^\alpha)\otimes \psi(u^\beta))(T_{P_{\alpha,\beta}}\otimes 1_{\A_X})$. But this follows because $T_{P_{\alpha,\beta}}$ is the linear map $T$ we get from applying Proposition~\ref{prop:psirep} to $u=u^\alpha$ and $v=u^\beta$, and then this proposition says that $T_{P_{\alpha,\beta}}$ is an intertwiner between $\psi(u^\alpha\otimes u^\beta)$ and $\psi(u^\alpha)\otimes \psi(u^\beta)$.

Finally for (5), fix $k\ge 1$, $\alpha,\beta\in \CC_k$, and $\varphi\in\CC_k(\alpha,\beta)$. Assume that $T_\varphi$ is an intertwiner between $u^\alpha$ and $u^\beta$. Since $u^{\Psi(\alpha)}=\psi(u^\alpha)$ and $u^{\Psi(\beta)}=\psi(u^\beta)$, we need to show that $T_{\Psi(\varphi)}$ between $\psi(u^\alpha)$ and $\psi(u^\beta)$. But this follows from applying Proposition~\ref{prop:psiobj} to $u=u^\alpha$, $v=u^\beta$, and $T=T_\varphi$. 

To show that $(\A_X,\{u^\alpha\}_{\alpha\in\TT})$ is $\RR_X$-admissible, we need a model $(\A_X,\{v^\beta\}_{\beta\in\RR_X})$ with $v^\alpha=u^\alpha$ for $\alpha\in\TT$. Since $\RR_X$ is the smallest completion of $\TT$, we know that all objects in $\RR_X$ are finite direct sums of subobjects of objects in $\TT$. If $\beta$ is a subobject of $\alpha\in \TT$ there is a projection $p\in B(\HH_\alpha)$ with $\HH_\beta=p\HH_\alpha$, and then we take $v^\beta:=(p\otimes 1_{\A_X})u^\alpha(p\otimes 1_{\A_X})$. For a direct sum of these subobjects we take the corresponding direct sum of representations. This gives a family $\{v^\beta\}_{\beta\in\RR_X}$, and straightforward calculations show that $(\A_X,\{v^\beta\}_{\beta\in\RR_X})$ is a model of $\RR_X$ with $v^\alpha=u^\alpha$ for $\alpha\in\TT$. We conclude that $(\A_X,\{u^\alpha\}_{\alpha\in\TT})$ is an $\RR_X$-admissible pair. 


It remains to see that $(\A_X,\{u^\alpha\}_{\alpha\in\TT})$ is universal. To see this, fix an $\RR_X$-admissible pair $(C,\{w^\alpha\}_{\alpha\in\TT}\})$. Recall from Notation~\ref{note:Psiof1} that $\psi_n:=\Psi^{n}(1)$. Let $c^{(0)}:=1_C$, and for each $n\geq 1$, let $c^{(n)} = (c_{u,v})_{u,v\in X^n}$ be given by $c^{(n)}:=w^{\psi_{n-1}}$. Then the elements $\{c_{u,v} : u,v\in X^n, n\geq 1\}$ satisfy the defining relations (1)--(3) for $\A_X$, and hence there exists a homomorphism $\varphi : \A_X \to C$ satisfying $\varphi(a_{u,v}) = c_{u,v}$ for each $u,v$ of the same length. This is equivalent to
\begin{equation}\label{eq: utowforpsisubn}
 (\id\otimes\varphi)u^{\psi_n} = w^{\psi_n}.
\end{equation}
We need to show that $(\id\otimes \varphi)u^\alpha = w^\alpha$ for all $\alpha\in\TT$. Fix $\alpha\in\TT$ and apply Lemma~\ref{lem:decomposingobjects} to get $S_\alpha\in\mathcal{T}(\alpha,\otimes_{j=1}^m\psi_{i_j})$ with $S_\alpha^*\circ S_\alpha=\id_\alpha$. Let $T_{S_\alpha}$ be the linear map associated to $S_\alpha$ as in Section~\ref{sec:C*tensorcat}. Then we use identity (iv) of Definition~\ref{def: model} and \eqref{eq: utowforpsisubn} to get
\begin{align*}
	(T_{S_\alpha}\otimes 1_C)w^{\alpha} &= \left(\bigotimes_{j=1}^m w^{\psi_{i_j}}\right)(T_{S_\alpha}\otimes 1_C) \\
	&= \left(\bigotimes_{j=1}^m (\id\otimes\varphi)u^{\psi_{i_j}}\right)(T_{S_\alpha}\otimes 1_C) \\
	&= (\id\otimes\varphi)\left(\bigotimes_{j=1}^m u^{\psi_{i_j}}\right)(\id\otimes\varphi)(T_{S_\alpha}\otimes 1_{\A_X}) \\
	&= (\id\otimes\varphi)\left(\bigotimes_{j=1}^m u^{\psi_{i_j}}(T_{S_\alpha}\otimes 1_{\A_X}))\right) \\
	&= (\id\otimes\varphi)\left((T_{S_\alpha}\otimes 1_{\A_X}) u^{\alpha} \right) \\
	&= (T_{S_\alpha}	\otimes 1_C)\left( (\id\otimes\varphi)u^{\alpha}\right).
\end{align*}
Since $S_\alpha^*\circ S_\alpha=\id_\alpha$, we know from Lemma~\ref{lem: propsofTs} that $T_{S_\alpha}\otimes 1_C$ is an isometry. Hence we must have $(\id\otimes\varphi)u^{\alpha}=w^{\alpha}$. It follows that $(\A_X,\{u^\alpha\}_{\alpha\in\TT})$ is universal. 
\end{proof}

\end{document}